\documentclass[10pt, reqno]{amsart}

\def\BibTeX{{\rm B\kern-.05em{\sc i\kern-.025em b}\kern-.08em
    T\kern-.1667em\lower.7ex\hbox{E}\kern-.125emX}}

\newtheorem{thm}{Theorem}[section]

\newtheorem{lem}[thm]{Lemma}
\newtheorem{prop}[thm]{Proposition}

\theoremstyle{definition}

\theoremstyle{remark}
\newtheorem{rem}{Remark}[section]

\numberwithin{equation}{section}

    \newcommand{\floor}[1]{\lfloor#1\rfloor}

    \newcommand{\EE}{\mathbb{E}}

    \renewcommand{\Pr}{\operatorname{P}}

    \newcommand{\dto}{\xrightarrow{d}}
    
    \newcommand{\vto}{\xrightarrow{v}}

    \newcommand{\eind}{\stackrel{d}{=}}
    \newcommand{\rmd}{\mathrm{d}}

\newcommand{\be}{\begin{equation}}
    \newcommand{\ee}{\end{equation}}

\begin{document}

\title[Maxima of linear processes] 
{Maxima of linear processes with heavy-tailed innovations and random coefficients}

%
\author{Danijel Krizmani\'{c}}

\address{Danijel Krizmani\'{c}\\ Department of Mathematics\\
        University of Rijeka\\
        Radmile Matej\v{c}i\'{c} 2, 51000 Rijeka\\
        Croatia}
\email{dkrizmanic@math.uniri.hr}





\begin{abstract}
We investigate maxima of linear processes with i.i.d.~heavy-tailed innovations and random coefficients. Using the point process approach we derive functional convergence of the partial maxima stochastic process in the space of non-decreasing c\`{a}dl\`{a}g functions on $[0,1]$ with the Skorohod $M_{1}$ topology.
\end{abstract}

\maketitle

\section{Introduction}
\label{intro}

Consider a strictly stationary sequence of
random variables $(X_{i})$ and denote by $M_{n} = \max \{X_{1}, X_{2}, \ldots, X_{n}\}$, $n \geq 1$, its accompanying sequence of partial
maxima. The principal concern of classical extreme value theory is with asymptotic distributional properties of the maximum $M_{n}$. It is well known that in the i.i.d.~case if there exist normalizing constants $a_{n}>0$ and $b_{n}$ such that
\begin{equation}\label{e:EVT}
  \Pr \Big( \frac{M_{n}-b_{n}}{a_{n}} \leq x \Big) \to G(x) \quad \textrm{as} \ n \to \infty,
 \end{equation}
 where $G$ is assumed non-degenerate, then $G$ necessarily belongs to the class of extreme value distributions (see for instance Gnedenko~\cite{Gn43} and Resnick~\cite{Re87}). In particular, (\ref{e:EVT}) holds with $G(x)=\exp \{-x^{-\alpha} \}$, $x>0$, for some $\alpha >0$, i.e.~the distribution of $X_{1}$ is in the domain of attraction of the Fr\'{e}chet distribution, if and only if $x \mapsto \Pr (X_{1} >x)$ is regularly varying at infinity with index $\alpha$, i.e.
 $$ \lim_{t \rightarrow \infty} \frac{\Pr(X_{1} >tx)}{\Pr(X_{1}>t)} = x^{-\alpha}$$
 for every $x>0$ (see Proposition 1.11 in Resnick~\cite{Re87}). A functional version of this is known to be true as well, the limit
process being an extremal process, and the convergence takes place
in the space of c\`{a}dl\`{a}g functions endowed with the Skorohod
$J_{1}$ topology. More precisely, relation (\ref{e:EVT}) is equivalent to
\begin{equation}\label{e:convPMP}
   M_{n}(\,\cdot\,) = \bigvee_{i=1}^{\lfloor n \cdot
   \rfloor}\frac{X_{i}}{a_{n}} \dto Y(\,\cdot\,)
 \end{equation}
in $D([0,1], \mathbb{R})$, the space of real-valued c\`{a}dl\`{a}g functions on
$[0,1]$, with the Skorohod $J_{1}$ topology, where $Y(\,\cdot\,)$
is an extremal process generated by $G$. If $G$ is the Fr\'{e}chet distribution, then $Y$ has
marginal distributions
$$\Pr (Y(t) \leq x) = e^{-tx^{-\alpha}}, \quad x \geq 0,\,t \in
[0,1].$$
This result was first proved by Lamperti~\cite{La64} (see also Resnick~\cite{Re87}, Proposition 4.20). For convenience we can put  $M_{n}(t)= X_{1}/a_{n}$ (or $M_{n}(t)=0$) for $t \in [0, 1/n)$.

In the dependent case, Adler~\cite{Ad78} obtained $J_{1}$ extremal functional convergence with the weak dependence condition similar to "asymptotic independence" condition introduced by Leadbetter~\cite{Le74}. For stationary sequences of jointly regularly varying random variables Basrak and Tafro~\cite{BaTa16} showed the invariance principle for the partial maximum process $M_{n}(\cdot)$ in $D([0,1], \mathbb{R})$ with the Skorohod $M_{1}$ topology (cf.~Krizmani\'{c}~\cite{Kr14}).

For a special class of weakly dependent random variables, the linear processes or moving averages processes with i.i.d.~heavy-tailed innovations (and deterministic coefficients), it is known that (\ref{e:convPMP}) holds, see for instance Resnick~\cite{Re87}, Proposition 4.28. In this paper we aim to obtain the functional convergence as in (\ref{e:convPMP}) for linear processes with random coefficients. Due to possible clustering of large values, the $J_{1}$ topology becomes inappropriate, and hence we will use the weaker Skorohod $M_{1}$ topology. In the proofs of our results we will use some methods and results which appear in Basrak and Krizmani\'{c}~\cite{BaKr14}, where they obtained functional convergence of partial sum processes with respect to Skorohod $M_{2}$ topology; Krizmani\'{c}~\cite{Kr18b}, where joint functional convergence of partial sums and maxima for linear processes was investigated and Krizmani\'{c}~\cite{Kr19}, where a functional limit theorem for sums of linear processes with heavy-tailed innovations and random coefficients was established.
For some related results on limit theory for sums of moving averages with random coefficients see Kulik~\cite{Ku06}.

In general, functional $M_{1}$ convergence of partial sum processes fails to hold. Clusters of large values in the sequence $(X_{n})$ may contain positive and negative values yielding the corresponding partial sum processes having jumps of opposite signs within temporal clusters of large values, and this precludes the $M_{1}$ convergence. For instance, this occurs for linear process with i.i.d.~heavy-tailed innovations $Z_{i}$ and deterministic coefficients $C_{0}=1$, $C_{1}=-1$, $C_{2}=1$ and $C_{i}=0$ for $i \geq 3$:
$$ X_{i}= Z_{i} - Z_{i-1} + Z_{i-2}, \qquad i \in \mathbb{Z}.$$
But in this case the convergence in distribution of the partial sum processes in the weaker $M_{2}$ topology can be shown to hold, see Krizmani\'{c}~\cite{Kr19}. For partial maxima processes we do not have similar problems with positive and negative values in clusters of big values since these processes are non-decreasing and thus only jumps with positive sign appear in them, which means one can have functional $M_{1}$ convergence.

 The paper is organized as follows. In Section~\ref{S:Pre} we introduce basic notions about linear processes, regular variation and Skorohod topologies. In Section~\ref{S:FLT} we derive functional convergence of the partial maxima stochastic process for finite order linear processes with i.i.d.~heavy-tailed innovations and random coefficients, and then we
extend this result to infinite order linear processes.

\section{Preliminaries}\label{S:Pre}

\emph{Linear processes}. Let $(Z_{i})_{i \in \mathbb{Z}}$ be a sequence of i.i.d.~random variables with regularly varying balanced tails, i.e.
\begin{equation}\label{e:regvar}
 \Pr(|Z_{i}| > x) = x^{-\alpha} L(x), \qquad x>0,
\end{equation}
for some $\alpha >0$ and slowly varying function $L$, and
\be\label{e:pr}
  \lim_{x \to \infty} \frac{\Pr(Z_i > x)}{\Pr(|Z_i| > x)}=p, \qquad
    \lim_{x \to \infty} \frac{\Pr(Z_i < -x)}{\Pr(|Z_i| > x)}=r,
\ee
where $p \in [0,1]$ and $p+r=1$.
Let $(a_{n})$ be a sequence of positive real numbers such that
\be\label{eq:niz}
n \Pr (|Z_{1}|>a_{n}) \to 1,
\ee
as $n \to \infty$. Then regular
variation of $Z_{i}$ can be expressed in terms of
vague convergence of measures on $\EE = \overline{\mathbb{R}} \setminus \{0\}$:
\begin{equation}
  \label{eq:onedimregvar}
  n \Pr( a_n^{-1} Z_i \in \cdot \, ) \vto \mu( \, \cdot \,) \qquad \textrm{as} \ n \to \infty,
\end{equation}
where $\mu$ is a measure on $\EE$  given by
\begin{equation}
\label{eq:mu}
  \mu(\rmd x) = \bigl( p \, 1_{(0, \infty)}(x) + r \, 1_{(-\infty, 0)}(x) \bigr) \, \alpha |x|^{-\alpha-1} \, \rmd x.
\end{equation}

We study linear processes with random coefficients, defined by
\begin{equation}\label{e:MArandom}
X_{i} = \sum_{j=0}^{\infty}C_{j}Z_{i-j}, \qquad i \in \mathbb{Z},
\end{equation}
where
$(C_{i})_{i \geq 0 }$ is a sequence of random variables independent of $(Z_{i})$ such that the above series is a.s.~convergent. One sufficient condition for that, which is commonly used in the literature is
\begin{equation}\label{e:momcond}
\sum_{j=0}^{\infty} \mathrm{E} |C_{j}|^{\delta} < \infty \qquad \textrm{for some}  \ \delta < \alpha,\,0 < \delta \leq 1.
\end{equation}

The regular variation property and Karamata's theorem imply $\mathrm{E}|Z_{1}|^{\beta} < \infty$ for every $\beta \in (0,\alpha)$ (cf.~Bingham et al.~\cite{BiGoTe89}, Proposition 1.5.10), which together with the moment condition (\ref{e:momcond}) yield the a.s.~convergence of the series in (\ref{e:MArandom}):
$$ \mathrm{E}|X_{i}|^{\delta} \leq \sum_{j=0}^{\infty} \mathrm{E}|C_{j}|^{\delta} \mathrm{E}|Z_{i-j}|^{\delta} = \mathrm{E}|Z_{1}|^{\delta} \sum_{j=0}^{\infty}\mathrm{E}|C_{j}|^{\delta} < \infty.$$
The same holds with the following moment conditions: if $\alpha < 1$ then there exists $\delta \in (0, \alpha)$ such that $\alpha + \delta <1$ and
$$ \sum_{j=0}^{\infty} \mathrm{E} |C_{j}|^{\alpha+\delta} < \infty, \qquad \sum_{j=0}^{\infty} \mathrm{E} |C_{j}|^{\alpha-\delta} < \infty,$$
and if $\alpha > 1$ then there exists $\delta >0$ such that
$$ \sum_{j=0}^{\infty} \mathrm{E} [|C_{j}|^{\alpha+\delta}]^{1/(\alpha+\delta)} < \infty, \qquad \sum_{j=0}^{\infty} \mathrm{E} [|C_{j}|^{\alpha-\delta}]^{1/(\alpha+\delta)} < \infty,$$
see Kulik~\cite{Ku06}.
Another condition that assures the a.s.~convergence of the series in the definition of linear processes in (\ref{e:MArandom}) with
$$ \begin{array}{rl}
                                   \mathrm{E}(Z_{1})=0, & \quad \textrm{if} \ \alpha >1,\\[0.4em]
                                   Z_{1} \ \textrm{is symmetric}, & \quad \textrm{if} \ \alpha =1,
                                 \end{array}$$
 and a.s.~bounded coefficients
 can be deduced from the results in Astrauskas~\cite{At83} for linear processes with deterministic coefficients:
$$ \sum_{j =0}^{\infty}c_{j}^{\alpha} L(c_{j}^{-1}) < \infty,$$
and additionally $\sum_{j}c_{j}^{2} < \infty$ for $\alpha >2$ and $\sum_{j}c_{j} ( 1 \vee L(c_{j}^{-1/2})) < \infty$ for $\alpha=2$,
where $(c_{j})$ is a sequence of positive real numbers such that $|C_{j}| \leq c_{j}$ a.s.~for all $j$, and $L$ as in (\ref{e:regvar}) (c.f.~Balan et al.~\cite{BJL16}).

\emph{Skorohod topologies}. We start by considering $D([0,1], \mathbb{R}^{d})$, the space of all right-continuous $\mathbb{R}^{d}$--valued functions on $[0,1]$ with left limits. For $x \in D([0,1],
\mathbb{R}^{d})$ the completed (thick) graph of $x$ is the set
\[
  G_{x}
  = \{ (t,z) \in [0,1] \times \mathbb{R}^{d} : z \in [[x(t-), x(t)]]\},
\]
where $x(t-)$ is the left limit of $x$ at $t$ and $[[a,b]]$ is the product segment, i.e.
$[[a,b]]=[a_{1},b_{1}] \times [a_{2},b_{2}] \ldots \times [a_{d},b_{d}]$
for $a=(a_{1}, a_{2}, \ldots, a_{d}), b=(b_{1}, b_{2}, \ldots, b_{d}) \in
\mathbb{R}^{d}$, and $[a_{i}, b_{i}]$ coincides with the closed interval $[a_{i} \wedge b_{i}, a_{i} \vee b_{i}]$, with $c \wedge d = \min \{c,d \}$ and $c \vee d = \max \{c, d\}$ for $c, d \in \mathbb{R}$. We define an
order on the graph $G_{x}$ by saying that $(t_{1},z_{1}) \le
(t_{2},z_{2})$ if either (i) $t_{1} < t_{2}$ or (ii) $t_{1} = t_{2}$
and $|x_{j}(t_{1}-) - z_{1j}| \le |x_{j}(t_{2}-) - z_{2j}|$
for all $j=1, 2, \ldots, d$. A weak parametric representation
of the graph $G_{x}$ is a continuous nondecreasing function $(r,u)$
mapping $[0,1]$ into $G_{x}$, with $r$ being the
time component and $u$ being the spatial component, such that $r(0)=0,
r(1)=1$ and $u(1)=x(1)$. Let $\Pi_{w}(x)$ denote the set of weak
parametric representations of the graph $G_{x}$. For $x_{1},x_{2}
\in D([0,1], \mathbb{R}^{d})$ define
\[
  d_{w}(x_{1},x_{2})
  = \inf \{ \|r_{1}-r_{2}\|_{[0,1]} \vee \|u_{1}-u_{2}\|_{[0,1]} : (r_{i},u_{i}) \in \Pi_{w}(x_{i}), i=1,2 \},
\]
where $\|x\|_{[0,1]} = \sup \{ \|x(t)\| : t \in [0,1] \}$. Now we
say that $x_{n} \to x$ in $D([0,1], \mathbb{R}^{d})$ for a sequence
$(x_{n})$ in the weak Skorohod $M_{1}$ (or shortly $WM_{1}$)
topology if $d_{w}(x_{n},x)\to 0$ as $n \to \infty$.

If we replace above the graph $G_{x}$ with the completed (thin) graph
\[
  \Gamma_{x}
  = \{ (t,z) \in [0,1] \times \mathbb{R}^{d} : z= \lambda x(t-) + (1-\lambda)x(t) \ \text{for some}\ \lambda \in [0,1] \},
\]
and weak parametric representations with strong parametric representations (i.e.~continuous nondecreasing functions $(r,u)$ mapping $[0,1]$ onto $\Gamma_{x}$), then we obtain the standard (or strong) Skorohod $M_{1}$ topology. This topology is induced by the metric
$$d_{M_{1}}(x_{1},x_{2})
  = \inf \{ \|r_{1}-r_{2}\|_{[0,1]} \vee \|u_{1}-u_{2}\|_{[0,1]} : (r_{i},u_{i}) \in \Pi_{s}(x_{i}), i=1,2 \},$$
where $\Pi_{s}(x)$ is the set of strong parametric representations of the graph $\Gamma_{x}$. The standard $M_{1}$ topology is stronger than the weak $M_{1}$ topology on $D([0,1],
\mathbb{R}^{d})$, but they coincide for $d=1$.

The $WM_{1}$ topology coincides with the topology induced by the metric
\begin{equation}\label{e:defdp}
 d_{p}(x_{1},x_{2})=\max \{ d_{M_{1}}(x_{1j},x_{2j}) :
j=1,\ldots,d\}
\end{equation}
 for $x_{i}=(x_{i1}, \ldots, x_{id}) \in D([0,1],
 \mathbb{R}^{d})$ and $i=1,2$. The metric $d_{p}$ induces the product topology on $D([0,1], \mathbb{R}^{d})$.

Using completed graphs and their parametric representations the Skorohod $M_{2}$ topology can also be defined. Here we give only its characterization by the Hausdorff metric on the space of graphs: for $x_{1},x_{2} \in D[0,1]$ the $M_{2}$ distance between $x_{1}$ and $x_{2}$ is given by
$$ d_{M_{2}}(x_{1}, x_{2}) = \bigg(\sup_{a \in \Gamma_{x_{1}}} \inf_{b \in \Gamma_{x_{2}}} d(a,b) \bigg) \vee \bigg(\sup_{a \in \Gamma_{x_{2}}} \inf_{b \in \Gamma_{x_{1}}} d(a,b) \bigg),$$
where $d$ is the metric on $\mathbb{R}^{2}$ defined by $d((x_{1},y_{1}),(x_{2},y_{2}))=|x_{1}-x_{2}| \vee |y_{1}-y_{2}|$ for $(x_{i},y_{i}) \in \mathbb{R}^{2},\,i=1,2$.
The metric $d_{M_{2}}$ induces the $M_{2}$ topology, which is weaker than the more frequently used $M_{1}$ topology. For more details and discussion on the $M_{1}$ and $M_{2}$ topologies we refer to Whitt~\cite{Whitt02}, sections 12.3-12.5.

Since the sample paths of the partial maximum process $M_{n}(\cdot)$ that we study in this paper are non-decreasing, we will restrict our attention to the subspace $D_{\uparrow}([0,1], \mathbb{R}^{d})$ of functions $x$ in $D[0,1], \mathbb{R}^{d})$ for which the coordinate functions $x_{i}$ are non-decreasing for all $i=1,\ldots,d$.

In the next section we will use the following two lemmas about the $M_{1}$ continuity of multiplication and maximum of two c\`{a}dl\`{a}g functions. The first one is based on Theorem 13.3.2 in Whitt~\cite{Whitt02}, and the second one follows easily from the fact that for monotone functions $M_{1}$ convergence is equivalent to point-wise convergence in a dense subset of $[0,1]$ including $0$ and $1$ (cf.~Corollary 12.5.1 in Whitt~\cite{Whitt02}). Denote by $\textrm{Disc}(x)$ the set of discontinuity points of $x \in D([0,1], \mathbb{R})$.


\begin{lem}\label{l:contmultpl}
Suppose that $x_{n} \to x$ and $y_{n} \to y$ in $D[0,1]$ with the $M_{1}$ topology. If for each $t \in \textrm{Disc}(x) \cap \textrm{Disc}(y)$, $x(t)$, $x(t-)$, $y(t)$ and $y(t-)$ are all nonnegative and $[x(t)-x(t-)][y(t)-y(t-)] \geq 0$, then $x_{n}y_{n} \to xy$ in $D[0,1]$ with the $M_{1}$ topology, where $(xy)(t) = x(t)y(t)$ for $t \in [0,1]$.
\end{lem}

\begin{lem}\label{l:weakM2transf}
The function $h \colon D_{\uparrow}([0,1], \mathbb{R}^{2}) \to D_{\uparrow}([0,1], \mathbb{R})$ defined by
$h(x,y)= x \vee y$, where
$$ (x \vee y)(t) = x(t) \vee y(t), \qquad t \in [0,1],$$
is continuous
when $D_{\uparrow}([0,1], \mathbb{R}^{2})$ is endowed with the weak $M_{1}$ topology and $D_{\uparrow}([0,1], \mathbb{R})$ is endowed with the standard $M_{1}$ topology.
\end{lem}

\section{Functional limit theorems}
\label{S:FLT}

Let $(Z_{i})_{i \in \mathbb{Z}}$ be an i.i.d.~sequence of regularly varying random variables with index $\alpha >0$, and $(C_{i})_{i \geq 0 }$ a sequence of random variables independent of $(Z_{i})$ such that the series defining the linear process
$$ X_{t} = \sum_{i=0}^{\infty}C_{i}Z_{t-i}, \qquad t \in \mathbb{Z},$$
is a.s.~convergent. Define the corresponding partial maximum process by
\be\label{eq:defWn}
M_{n}(t) = \left\{ \begin{array}{cc}
                                   \displaystyle \frac{1}{a_{n}} \bigvee_{i=1}^{\floor {nt}}X_{i}, & \quad  \displaystyle t \geq \frac{1}{n},\\[1.2em]
                                   \displaystyle \frac{X_{1}}{a_{n}}, & \quad \displaystyle  t < \frac{1}{n},
                                 \end{array}\right.
\ee
for $t \in [0,1]$,
with the normalizing sequence $(a_n)$ as in~\eqref{eq:niz}. Let
\begin{equation}\label{e:Cplusminus}
C_{+}= \max \{ C_{j} \vee 0 : j \geq 0 \} \qquad \textrm{and} \qquad C_{-}= \max \{ -C_{j} \vee 0 : j \geq 0 \}.
\end{equation}
Before the main theorem we have two auxiliary results. Define the maximum functional
$ \Phi \colon \mathbf{M}_{p}([0,1] \times \EE) \to D_{\uparrow}([0,1], \mathbb{R}^{2})$
by
$$ \Phi \Big(\sum_{i}\delta_{(t_{i}, x_{i})} \Big) (t)
  =  \Big(  \bigvee_{t_{i} \leq t}|x_{i}| 1_{\{ x_{i}>0 \}},  \bigvee_{t_{i} \leq t} |x_{i}| 1_{\{ x_{i}<0 \}} \Big)$$
  for $t \in [0,1]$
(with the convention $\vee \emptyset = 0$), where the space $\mathbf{M}_p([0,1] \times \EE)$ of Radon point
measures on $[0,1] \times \EE$ is equipped with the vague
topology (see Chapter 3 in Resnick~{Re87}).

\begin{prop}\label{p:contfunct}
The maximum functional $\Phi \colon \mathbf{M}_{p}([0,1] \times \EE) \to D_{\uparrow}([0,1], \mathbb{R}^{2}) $ is continuous on the set
\begin{eqnarray*}
  \Lambda &=&  \{ \eta \in \mathbf{M}_{p}([0,1] \times \EE) :
   \eta ( \{0,1 \} \times \EE) = \eta ([0,1] \times \{ \pm \infty \}) = 0 \},
\end{eqnarray*}
when $D_{\uparrow}([0,1], \mathbb{R}^{2})$ is endowed with the weak $M_{1}$ topology
\end{prop}
\begin{proof}
Take an arbitrary $\eta \in \Lambda$ and suppose that $\eta_{n} \vto \eta$ as $n \to \infty$ in $\mathbf{M}_p([0,1] \times
\EE)$. We need to show that
$\Phi(\eta_n) \to \Phi(\eta)$ in $D_{\uparrow}([0,1],
\mathbb{R}^{2})$ according to the $WM_1$ topology. By Theorem 12.5.2 in Whitt~\cite{Whitt02}, it suffices to prove that,
as $n \to \infty$,
$$ d_{p}(\Phi(\eta_{n}), \Phi(\eta)) =
\max_{k=1, 2}d_{M_{1}}(\Phi_{k}(\eta_{n}),
\Phi_{k}(\eta)) \to 0.$$
Let
$$T= \{ t \in [0,1] : \eta (\{t\} \times \EE) = 0 \}.$$
Since $\eta$ is a Radon point measure, the set $T$ is dense in $[0,1]$. Fix $t \in T$ and take $\epsilon >0$ such that $\eta([0,t] \times \{\pm \epsilon\})=0$. Since $\eta$ is a Radon point measure, we can arrange that, letting $\epsilon \downarrow 0$, the convergence to $0$ is through a sequence of values $(\epsilon_{j})$ such that $\eta([0,t] \times \{\pm \epsilon_{j}\})=0$ for all $j \in \mathbb{N}$. For $u>0$ let $\mathbb{E}_{u} = \mathbb{E} \setminus (-u,u)$. Since the set $[0,t] \times \EE_{\epsilon}$ is relatively compact in
$[0,1] \times \EE$, there exists a nonnegative integer
$k=k(\eta)$ such that
$$ \eta ([0,t] \times \EE_{\epsilon}) = k < \infty.$$
By assumption, $\eta$ does not have any atoms on the border of the
set $[0,t] \times \EE_{\epsilon}$, and therefore by Lemma 7.1 in Resnick~\cite{Resnick07} there exists a positive integer $n_{0}$ such
that
$$ \eta_{n} ([0,t] \times \EE_{\epsilon})=k \qquad \textrm{for all} \ n \geq n_{0}.$$
Let
$(t_{i},x_{i})$ for $i=1,\ldots,k$ be the atoms of $\eta$ in
$[0,t] \times \EE_{\epsilon}$. By the same lemma, the $k$ atoms
$(t_{i}^{(n)}, x_{i}^{(n)})$ of $\eta_{n}$ in $[0,t] \times \EE_{\epsilon}$ (for $n \geq n_{0}$) can be labeled in such a way that
for every $i \in \{1,\ldots,k\}$ we have
$$ (t_{i}^{(n)}, x_{i}^{(n)}) \to (t_{i},x_{i}) \qquad \textrm{as}
\ n \to \infty.$$ In particular, for any $\delta >0$ there exists a
positive integer $n_{\delta} \geq n_{0}$ such that for all $n \geq
n_{\delta}$,
\begin{equation*}\label{e:etaconv}
 |t_{i}^{(n)} - t_{i}| < \delta \quad \textrm{and} \quad
 |x_{i}^{(n)}- x_{i}| < \delta \qquad \textrm{for} \ i=1,\ldots,k.
\end{equation*}
If $k=0$, then (for large $n$) the atoms of $\eta$ and $\eta_{n}$ in $[0,t] \times \EE$ are all situated in $[0,t] \times (-\epsilon, \epsilon)$. Hence
$ \Phi_{1}(\eta)(t) \in [0,\epsilon)$ and $ \Phi_{1}(\eta_{n})(t) \in [0, \epsilon)$, which imply
\begin{equation}\label{e:conv1}
  |\Phi_{1}(\eta_{n})(t) - \Phi_{1}(\eta)(t)| < \epsilon.
\end{equation}
If $k \geq 1$, take $\delta = \epsilon$. Note that $|x_{i}^{(n)}-x_{i}| < \delta$ implies $x_{i}^{(n)} >0$ iff $x_{i} >0$. Hence we have
\begin{eqnarray}\label{e:conv2}
  \nonumber |\Phi_{1}(\eta_{n})(t) - \Phi_{1}(\eta)(t)| &=& \bigg| \bigvee_{i=1}^{k}|x_{i}^{(n)}| 1_{\{ x_{i}^{(n)}>0\}} - \bigvee_{i=1}^{k}|x_{i}| 1_{\{ x_{i}>0\}} \bigg|\\[0.5em]
   \nonumber &\leq& \bigvee_{i=1}^{k} \Big| (|x_{i}^{(n)}|-|x_{i}|) 1_{\{ x_{i}>0\}} \Big|\\[0.5em]
    &  \leq &  \bigvee_{i=1}^{k}
    |x_{i}^{(n)}-x_{i}| < \epsilon,
\end{eqnarray}
    where the first inequality above follows from the elementary inequality
 \begin{equation*}\label{e:maxineq}
 \Big| \bigvee_{i=1}^{k}a_{i} - \bigvee_{i=1}^{k}b_{i} \Big| \leq
\bigvee_{i=1}^{k}|a_{i}-b_{i}|,
 \end{equation*}
which holds for arbitrary real numbers $a_{1}, \ldots, a_{k}, b_{1},
\ldots, b_{k}$.
Therefore from (\ref{e:conv1}) and (\ref{e:conv2}) we obtain
 $$\limsup_{n \to \infty}|\Phi_{1}(\eta_{n})(t) -
 \Phi_{1}(\eta)(t)| \leq \epsilon,$$
  and letting $\epsilon \to 0$ it follows that
 $\Phi_{1}(\eta_{n})(t) \to \Phi_{1}(\eta)(t)$ as $n \to
 \infty$. Since $\Phi_{1}(\eta)$ and $\Phi_{1}(\eta_{n})$ are nondecreasing functions, and by Corollary 12.5.1 in Whitt~\cite{Whitt02} $M_{1}$ convergence for monotone functions is equivalent to point-wise convergence in a dense subset of points plus convergence at the endpoints, we conclude that $d_{M_{1}}(\Phi_{1}(\eta_{n}),
 \Phi_{1}(\eta)) \to 0$ as $n \to \infty$. In the same manner we obtain $d_{M_{1}}(\Phi_{2}(\eta_{n}),
 \Phi_{2}(\eta)) \to 0$, and therefore we conclude that $\Phi$ is continuous at $\eta$.
\end{proof}

\begin{prop}\label{p:FLT}
Let $(X_{i})$ be a linear process defined by
$$ X_{i} = \sum_{j=0}^{\infty}C_{j}Z_{i-j}, \qquad i \in \mathbb{Z},$$
where $(Z_{i})_{i \in \mathbb{Z}}$ is an i.i.d.~sequence of random variables satisfying $(\ref{e:regvar})$ and $(\ref{e:pr})$ with $\alpha >0$, and
 $(C_{i})_{i \geq 0 }$ is a sequence of random variables independent of $(Z_{i})$ such that the series defying the above linear process is a.s.~convergent.
Let $$ W_{n}(t) := \bigvee_{i=1}^{\floor {nt}}\frac{ |Z_{i}|}{a_{n}}(C_{+}1_{\{ Z_{i} >0 \}} + C_{-} 1_{\{ Z_{i}<0 \}}), \qquad t \in [0,1],$$
with $C_{+}$ and $C_{-}$ defined in $(\ref{e:Cplusminus})$.
Then, as $n \to \infty$,
\begin{equation}\label{e:pomkonv}
 W_{n}(\,\cdot\,) \dto  C^{(1)}W^{(1)}(\,\cdot\,) \vee C^{(2)}W^{(2)}(\,\cdot\,)
\end{equation}
 in $D_{\uparrow}[0,1] :=D_{\uparrow}([0,1], \mathbb{R})$ with the $M_{1}$ topology,
where $W^{(1)}$ and $W^{(2)}$ are extremal processes with exponent measures
$p \alpha x^{-\alpha-1}1_{(0,\infty)}(x)\,dx$ and $r \alpha x^{-\alpha-1}1_{(0,\infty)}(x)\,dx$ respectively, with $p$ and $r$ defined in $(\ref{e:pr})$, and $(C^{(1)}, C^{(2)})$ is a random vector, independent of $(W^{(1)}, W^{(2)})$, such that
$(C^{(1)}, C^{(2)}) \eind (C_{+}, C_{-})$.
\end{prop}

\begin{rem}
In Proposition~\ref{p:FLT}, as well as in the sequel of this paper, we suppose $W^{(1)}$ is an extremal process if $p>0$, and a zero process if $p=0$. Analogously for $W^{(2)}$.
\end{rem}

\begin{proof}[Proof of Proposition~\ref{p:FLT}]
Since the random variables $Z_{i}$ are i.i.d.~and regularly varying, Corollary 6.1 in Resnick~\cite{Resnick07} yields
\begin{equation}\label{e:ppconv}
 N_{n} := \sum_{i=1}^{n}\delta_{(\frac{i}{n}, \frac{Z_{i}}{a_{n}})} \dto N := \sum_{i}\delta_{(t_{i},j_{i})} \qquad \textrm{as} \ n \to \infty,
\end{equation}
in $\mathbf{M}_{p}([0,1] \times \EE)$, where the limiting point process $N$ is a Poisson process with intensity measure $\emph{Leb} \times \mu$, with $\mu$ as in (\ref{eq:mu}). Since $P( N \in \Lambda)=1$ (cf.~Resnick~\cite{Resnick}, p.~221) from (\ref{e:ppconv}) by an application of Proposition~\ref{p:contfunct} and the continuous mapping theorem (see for instance Theorem 3.1 in Resnick~\cite{Resnick07}) we obtain
$ \Phi(N_{n})(\,\cdot\,) \dto \Phi(N)(\,\cdot\,)$ as $n \to \infty$, i.e.
\begin{eqnarray}\label{e:conv11}
 \nonumber (W_{n}^{(1)}(\,\cdot\,), W_{n}^{(2)}(\,\cdot\,)) & := & \Big( \bigvee_{i=1}^{\floor{n\,\cdot}} \frac{|Z_{i}|}{a_{n}} 1_{ \{ Z_{i} > 0 \} }, \bigvee_{i=1}^{\floor{n\,\cdot}}\frac{|Z_{i}|}{a_{n}} 1_{\{ Z_{i}<0 \}}  \Big)\\[0.6em]
  & \hspace*{-10em} \dto & \hspace*{-5em} (W^{(1)}(\,\cdot\,), W^{(2)}(\,\cdot\,)) :=
 \Big(  \bigvee_{t_{i} \leq\,\cdot} |j_{i}| 1_{\{ j_{i} >0 \}}, \bigvee_{t_{i} \leq\,\cdot} |j_{i}| 1_{\{ j_{i}<0 \}} \Big)
\end{eqnarray}
in $D_{\uparrow}([0,1], \mathbb{R}^{2})$ under the weak $M_{1}$ topology.

The space $D([0,1],\mathbb{R})$ equipped with the Skorokhod $J_{1}$ topology is a Polish space (i.e.~metrizable as a complete separable metric space), see Section 14 in Billingsley~\cite{Bi68}, and therefore the same holds for the (standard) $M_{1}$ topology, since it is topologically complete (see Section 12.8 in Whitt~\cite{Whitt02}) and separability remains preserved in the weaker topology. The space $D_{\uparrow}[0,1]$ is a closed subspace of $D([0,1], \mathbb{R})$ (cf.~Lemma 13.2.3 in Whitt~\cite{Whitt02}), and hence also Polish. Further, the space $D_{\uparrow}([0,1], \mathbb{R}^{2})$ equipped with the weak $M_{1}$ topology is separable as a direct product of two separable topological spaces. It is also topologically complete since the product metric in (\ref{e:defdp}) inherits the completeness of the component metrics. Thus we conclude that $D_{\uparrow}([0,1], \mathbb{R}^{2})$ with the weak $M_{1}$ topology is also a Polish space,
and hence by Corollary 5.18 in Kallenberg~\cite{Ka97}, we can find a random vector ($C^{(1)}, C^{(2)})$, independent of $(W^{(1)}, W^{(2)})$, such that
\begin{equation}\label{e:eqdistrC}
(C^{(1)}, C^{(2)}) \eind (C_{+}, C_{-}).
\end{equation}
This, relation (\ref{e:conv11}) and the fact that $(C_{+}, C_{-})$ is independent of $(W_{n}^{(1)}, W_{n}^{(2)})$, by an application of Theorem 3.29 in Kallenberg~\cite{Ka97}, imply that, as $n \to \infty$,
  \begin{equation}\label{e:zajedkonvK}
   (B^{1}(\,\cdot\,), B^{2}(\,\cdot\,), W_{n}^{1}(\,\cdot\,), W_{n}^{2}(\,\cdot\,)) \dto (B^{(1)}(\,\cdot\,), B^{(2)}(\,\cdot\,), W^{(1)}(\,\cdot\,), W^{(2)}(\,\cdot\,))
  \end{equation}
  in $D_{\uparrow}([0,1], \mathbb{R}^{4})$ with the product $M_{1}$ topology, where $B^{1}(t)=C_{+}$, $B^{2}(t)=C_{-}$, $B^{(1)}(t)=C^{(1)}$ and $B^{(2)}(t)=C^{(2)}$ for $t \in [0,1]$.

Let $g \colon D_{\uparrow}([0,1], \mathbb{R}^{4}) \to D_{\uparrow}([0,1], \mathbb{R}^{2})$ be a function defined by
$$ g(x) = (x_{1}x_{3}, x_{2}x_{4}), \qquad x=(x_{1},x_{2}, x_{3}, x_{4}) \in D_{\uparrow}([0,1], \mathbb{R}^{4}).$$
 Denote by $\widetilde{D}_{1,2}$ the set of all functions in $D_{\uparrow}([0,1], \mathbb{R}^{4})$ for which the first two component functions have no common discontinuity points, i.e.
$$ \widetilde{D}_{1,2} = \{ (u,v,z,w) \in D_{\uparrow}([0,1], \mathbb{R}^{4}) : \textrm{Disc}(u) = \textrm{Disc}(v)= \emptyset \}.$$
By Lemma~\ref{l:contmultpl} the function $g$ is continuous on the set $\widetilde{D}_{1,2}$ in the weak $M_{1}$ topology, and hence $\textrm{Disc}(g) \subseteq \widetilde{D}_{1,2}^{c}$. Denoting
$\widetilde{D}_{1} = \{ u \in D_{\uparrow}[0,1] : \textrm{Disc}(u)= \emptyset \}$ we obtain
\begin{eqnarray*}
\Pr[ (B^{(1)}, B^{(2)}, W^{(1)}, W^{(2)}) \in \textrm{Disc}(g) ] & \leq & \Pr [ (B^{(1)}, B^{(2)}, W^{(1)}, W^{(2)})  \in \widetilde{D}_{1,2}^{c}]\\[0.5em]
 &  \leq &  \Pr [\{ B^{(1)} \in \widetilde{D}_{1}^{c} \} \cup \{ B^{(2)} \in \widetilde{D}_{1}^{c} \}]=0,
 \end{eqnarray*}
where the last equality holds since $B^{(1)}$ and $B^{(2)}$ have no discontinuity points. This allows us to apply the continuous mapping theorem to relation (\ref{e:zajedkonvK}) yielding
$g(B^{1}, B^{2}, W_{n}^{1}, W_{n}^{2}) \dto g(B^{(1)}, B^{(2)}, W^{(1)}, W^{(2)})$, i.e.
\begin{equation}\label{e:zajedkonvK2}
 (C_{+} W_{n}^{1}, C_{-} W_{n}^{2}) \dto (C^{(1)} W^{(1)}, C^{(2)} W^{(2)}) \qquad \textrm{as} \ n \to \infty,
\end{equation}
 in $D_{\uparrow}([0,1], \mathbb{R}^{2})$ with the weak $M_{1}$ topology. Now from (\ref{e:zajedkonvK2}) by Lemma~\ref{l:weakM2transf} and the continuous mapping theorem it follows
 $ C_{+} W_{n}^{1} \vee C_{-} W_{n}^{2} \dto C^{(1)} W^{(1)} \vee C^{(2)} W^{(2)}$ as $n \to \infty$, i.e.
 $$ \bigvee_{i=1}^{\floor {n\,\cdot}}\frac{ C_{+}|Z_{i}|}{a_{n}} 1_{\{ Z_{i} >0 \}} \vee  \bigvee_{i=1}^{\floor {n\,\cdot}}\frac{ C_{-}|Z_{i}|}{a_{n}} 1_{\{ Z_{i}<0 \}} \dto
  \bigvee_{t_{i} \leq\,\cdot} C^{(1)}|j_{i}| 1_{\{ j_{i} >0 \}} \vee \bigvee_{t_{i} \leq\,\cdot} C^{(2)}|j_{i}| 1_{\{ j_{i}<0 \}}$$
  in $D_{\uparrow}[0,1]$ with the $M_{1}$ topology. This is in fact (\ref{e:pomkonv}) since the process in the converging sequence in the last relation is equal to $W_{n}(\,\cdot\,)$.
It remains only to show that the corresponding limiting process is of the form claimed in the statement of the proposition. Denote it by $M(\,\cdot\,)$.  By an application of Proposition 3.7 in Resnick~\cite{Re87} we obtain that the restricted processes $\sum_{i}\delta_{(t_{i}, j_{i} 1_{\{j_{i} >0 \}})}$ and $\sum_{i}\delta_{(t_{i}, -j_{i} 1_{\{j_{i} <0 \}})}$ are independent Poisson processes with intensity measures $\emph{Leb} \times \mu_{+}$ and $\emph{Leb} \times \mu_{-}$ respectively, where
$$  \mu_{+}(\rmd x) = p \, 1_{(0, \infty)}(x) \, \alpha x^{-\alpha-1} \, \rmd x
\qquad \textrm{and} \qquad  \mu_{-}(\rmd x) =  r \, 1_{(0, \infty)}(x) \, \alpha x^{-\alpha-1} \, \rmd x.$$
(cf.~Theorem 5.2 in Last and Penrose~\cite{LaPe18}). From this we conclude that the processes
$$ W^{(1)}(\,\cdot\,) =  \bigvee_{t_{i} \leq\,\cdot} |j_{i}| 1_{\{ j_{i} >0 \}} \qquad \textrm{and} \qquad W^{(2)}(\,\cdot\,)=\bigvee_{t_{i} \leq\,\cdot} |j_{i}| 1_{\{ j_{i}<0 \}}$$
are extremal processes with exponent measures $\mu_{+}$ and $\mu_{-}$ respectively (see Resnick~\cite{Re87},  Section 4.3; Resnick~\cite{Resnick07}, p.~161), and hence
$M(t)=  C^{(1)}W^{(1)}(t) \vee C^{(2)}W^{(2)}(t)$ for $ t \in [0,1]$.
\end{proof}

In deriving functional convergence of the partial maxima process we first deal with finite order linear processes. Fix $q \in \mathbb{N}$ and let
$$ X_{t} = \sum_{i=0}^{q}C_{i}Z_{t-i}, \qquad t \in \mathbb{Z}.$$
In this case $C_{+}$ and $C_{-}$ reduce to
$C_{+}= \max \{ C_{j} \vee 0 : j=0, \ldots, q\}$ and $C_{-}= \max \{ -C_{j} \vee 0 : j=0, \ldots, q \}$. Denote by $M$ the limiting process in Proposition~\ref{p:FLT}, i.e.
\begin{equation}\label{e:limprocess}
 M(\,\cdot\,) = C^{(1)}W^{(1)}(\,\cdot\,) \vee C^{(2)}W^{(2)}(\,\cdot\,),
\end{equation}
where $W^{(1)}$ is an extremal process with exponent measure $\mu_{+}(\rmd x) = p \alpha x^{-\alpha-1} \rmd x$ for $x>0$, $W^{(2)}$ is an extremal process with exponent measure $\mu_{-}(\rmd x) = r \alpha x^{-\alpha-1} \rmd x$ for $x>0$, and $(C^{(1)}, C^{(2)})$ is a two dimensional random vector, independent of $(W^{(1)}, W^{(2)})$, such that
$(C^{(1)}, C^{(2)}) \eind (C_{+}, C_{-})$. Taking into account the proof of Proposition~\ref{p:FLT} observe that
  $$ M(t) = \bigvee_{t_{i} \leq t}|j_{i}| ( C^{(1)}1_{\{ j_{i} >0 \}} + C^{(2)} 1_{\{ j_{i}<0 \}}), \qquad t \in [0,1],$$
  where $\sum_{i}\delta_{(t_{i},j_{i})}$ is a Poisson process with intensity measure $\emph{Leb} \times \mu$, with $\mu$ as in (\ref{eq:mu}).

\begin{thm}\label{t:FLT}
Let $(Z_{i})_{i \in \mathbb{Z}}$ be an i.i.d.~sequence of random variables satisfying $(\ref{e:regvar})$ and $(\ref{e:pr})$ with $\alpha >0$.
Assume $C_{0}, C_{1}, \ldots, C_{q}$ are random variables independent of $(Z_{i})$. Then, as $n \to \infty$,
$$ M_{n}(\,\cdot\,) \dto  M(\,\cdot\,)$$
in $D_{\uparrow}[0,1]$ endowed with the $M_{1}$ topology.
\end{thm}

\begin{proof}
Our aim is to show  that for every $\delta >0$
$$ \lim_{n \to \infty} \Pr[d_{M_{1}}(W_{n}, M_{n}) > \delta]=0,$$
since then from Proposition~\ref{p:FLT} by an application of Slutsky's theorem (see for instance Theorem 3.4 in Resnick~\cite{Resnick07}) we will obtain $M_{n}(\,\cdot\,) \dto M(\,\cdot\,)$ as $n \to \infty$
in $D_{\uparrow}[0,1]$ endowed with the $M_{1}$ topology. It suffices to show that
\begin{equation}\label{e:max1}
\lim_{n \to \infty} \Pr[d_{M_{2}}(W_{n}, M_{n}) > \delta]=0.
\end{equation}
Indeed, by Remark 12.8.1 in Whitt~\cite{Whitt02} the following metric is a complete metric topologically equivalent to $d_{M_{1}}$:
$$ {d_{M_{1}}^{*}}(x_{1}, x_{2}) = d_{M_{2}}(x_{1}, x_{2}) + \lambda (\widehat{\omega}(x_{1},\cdot), \widehat{\omega}(x_{2},\cdot)),$$
where $\lambda$ is the L\'{e}vy metric on a space of distributions
$$ \lambda (F_{1},F_{2}) = \inf \{ \epsilon >0 : F_{2}(x-\epsilon) - \epsilon \leq F_{1}(x) \leq F_{2}(x+\epsilon) + \epsilon \ \ \textrm{for all} \ x \}$$
and
$$ \widehat{\omega}(x,z) = \left\{ \begin{array}{cc}
                                   \omega(x,e^{z}), & \quad z<0,\\[0.4em]
                                   \omega(x,1), & \quad z \geq 0,
                                 \end{array}\right.$$
with
$ \omega (x,\rho) = \sup_{0 \leq t \leq 1} \omega (x, t, \rho)$ and
\begin{equation*}\label{e:oscillationf}
\omega(x,t,\rho) = \sup_{0 \vee (t-\rho) \leq t_{1} < t_{2} < t_{3} \leq (t+\rho) \wedge 1} ||x(t_{2}) - [x(t_{1}), x(t_{3})]||
\end{equation*}
where $\rho>0$ and $\|z-A\|$ denotes the distance between a point $z$ and a subset $A \subseteq \mathbb{R}$.
Since $W_{n}(\,\cdot\,)$ and $M_{n}(\,\cdot\,)$ are nondecreasing, for $t_{1} < t_{2} < t_{3}$ it holds that
$ \|W_{n}(t_{2}) - [W_{n}(t_{1}), W_{n}(t_{3})] \|=0$, which implies $\omega(W_{n}, \rho)=0$ for all $\rho>0$, and similarly $\omega(M_{n}, \rho)=0$. Hence $\lambda (W_{n}, M_{n})=0$, and $d_{M_{1}}^{*}(W_{n}, M_{n}) = d_{M_{2}}(W_{n}, M_{n})$.

In order to show (\ref{e:max1}) fix $\delta >0$ and let $n \in \mathbb{N}$ be large enough, i.e.~$n > \max\{2q, 2q/\delta \}$.
Then by the definition of the metric $d_{M_{2}}$ we have
\begin{equation*}
  d_{M_{2}}(W_{n}, M_{n}) = \bigg(\sup_{v \in \Gamma_{W_{n}}} \inf_{z \in \Gamma_{ M_{n}}} d(v,z) \bigg) \vee \bigg(\sup_{v \in \Gamma_{ M_{n}}} \inf_{z \in \Gamma_{W_{n}}} d(v,z) \bigg)  = : Y_{n} \vee T_{n}.
\end{equation*}
Hence
\be\label{eq:AB}
\Pr [d_{M_{2}}(W_{n}, M_{n})> \delta ] \leq \Pr(Y_{n}>\delta) + \Pr(T_{n}>\delta).
\ee
Now, we estimate the first term on the right hand side of (\ref{eq:AB}).
Let
$$ D_{n} = \{\exists\,v \in \Gamma_{W_{n}} \ \textrm{such that} \ d(v,z) > \delta \ \textrm{for every} \ z \in \Gamma_{ M_{n}} \},$$
and note that by the definition of $Y_{n}$
\begin{equation}\label{e:Yn}
 \{Y_{n} > \delta\} \subseteq D_{n}.
\end{equation}
On the event $D_{n}$ it holds that $d(v, \Gamma_{M_{n}})> \delta$. Let $v=(t_{v},x_{v})$. We claim that
\begin{equation}\label{e:i*}
 \Big| W_{n} \Big( \frac{i^{*}}{n} \Big) - M_{n} \Big( \frac{i^{*}}{n} \Big) \Big| > \delta,
\end{equation}
where $i^{*}=\floor{nt_{v}}$ or $i^{*}=\floor{nt_{v}}-1$. To see this, observe that $t_{v} \in [i/n, (i+1)/n)$ for some $i \in \{1,\ldots,n-1\}$ (or $t_{v}=1$). If $x_{v} = W_{n}(i/n)$ (i.e.~$v$ lies on a horizontal part of the completed graph), then clearly
$$\Big| W_{n} \Big( \frac{i}{n} \Big) -  M_{n} \Big( \frac{i}{n} \Big) \Big| \geq d(v,  \Gamma_{ M_{n}}) > \delta,$$
and we put $i^{*}=i$.
 On the other hand, if $x_{v} \in [W_{n}((i-1)/n), W_{n}(i/n))$ (i.e.~$v$ lies on a vertical part of the completed graph), one can similarly show that
 $$ \Big| W_{n} \Big(\frac{i-1}{n} \Big) -  M_{n} \Big(\frac{i-1}{n} \Big) \Big| > \delta \qquad \textrm{if} \ M_{n} \Big( \frac{i-1}{n} \Big) > x_{v},$$
 and
 $$ \Big| W_{n} \Big(\frac{i}{n} \Big) -  M_{n} \Big(\frac{i}{n} \Big) \Big| > \delta \qquad \textrm{if} \ M_{n} \Big( \frac{i-1}{n} \Big) < x_{v}.$$
In the first case put $i^{*}=i-1$ and in the second $i^{*}=i$.
Since $i= \floor{nt_{v}}$ we conclude that (\ref{e:i*}) holds.
Moreover, since $|i^{*}/n -(i^{*}+l)/n| \leq q/n < \delta$ for every $l=1,\ldots,q$ (such that $i^{*}+l \leq n$), from the definition of the set $D_{n}$ one can similarly conclude that
\begin{equation}\label{e:i*q}
 \Big| W_{n} \Big( \frac{i^{*}}{n} \Big) - M_{n} \Big( \frac{i^{*}+l}{n} \Big) \Big| > \delta.
\end{equation}
Let $C = C_{+} \vee C_{-}$.
We claim that
\begin{equation}\label{e:estim1}
 D_{n} \subseteq H_{n, 1} \cup H_{n, 2} \cup H_{n, 3},
\end{equation}
where
\begin{eqnarray}
 \nonumber H_{n, 1} & = & \bigg\{ \exists\,l \in \{-q,\ldots,q\} \cup \{n-q+1, \ldots, n\} \ \textrm{such that} \ \frac{ C|Z_{l}|}{a_{n}} > \frac{\delta}{4(q+1)} \bigg\},\\[0.1em]
  \nonumber H_{n, 2} & = & \bigg\{ \exists\,k \in \{1, \ldots, n\} \ \textrm{and} \ \exists\,l \in \{k-q,\ldots,k+q\} \setminus \{k\} \ \textrm{such that}\\[0.1em]
  \nonumber & & \ \frac{ C|Z_{k}|}{a_{n}} > \frac{\delta}{4(q+1)} \ \textrm{and} \  \frac{ C|Z_{l}|}{a_{n}} > \frac{\delta}{4(q+1)}  \bigg\},\\[0.1em]
  \nonumber H_{n, 3} & = & \bigg\{ \exists\,k \in \{1, \ldots, n\}, \ \exists\,j \in \{1,\ldots,n\} \setminus \{k,\ldots,k+q\}, \ \exists\,l_{1} \in \{0,\ldots,q\}\\[0.1em]
  \nonumber & & \textrm{and} \ \exists\,l \in \{0,\ldots,q\} \setminus \{l_{1}\} \ \textrm{such that} \ \frac{ C|Z_{k}|}{a_{n}} > \frac{\delta}{4(q+1)},\\[0.1em]
  \nonumber & & \ \frac{ C|Z_{j-l_{1}}|}{a_{n}} > \frac{\delta}{4(q+1)} \ \textrm{and} \  \frac{ C|Z_{j-l}|}{a_{n}} > \frac{\delta}{4(q+1)}  \bigg\}.
\end{eqnarray}
Relation (\ref{e:estim1}) will be proven if we show that
$$ \widehat{D}_{n} := D_{n} \cap (H_{n, 1} \cup H_{n, 2})^{c} \subseteq H_{n, 3}.$$
Assume the event $\widehat{D}_{n}$ occurs. Then necessarily $W_{n}(i^{*}/n) > \delta / [4(q+1)]$. Indeed, if $W_{n}(i^{*}/n) \leq \delta / [4(q+1)]$, i.e.
$$ \bigvee_{j=1}^{i^{*}}\frac{ |Z_{j}|}{a_{n}} \big(C_{+} 1_{\{ Z_{j}>0 \}} + C_{-} 1_{\{ Z_{j}<0 \}} \big)   = W_{n} \Big( \frac{i^{*}}{n} \Big) \leq \frac{\delta}{4(q+1)},$$
then for every $s \in \{ q+1, \ldots, i^{*}\}$ it holds that
\begin{eqnarray}\label{e:phipm}
   \nonumber \frac{X_{s}}{a_{n}}  &=&  \sum_{j=0}^{q} \frac{C_{j} Z_{s-j}}{a_{n}} \leq  \sum_{j=0}^{q}  \frac{|Z_{s-j}|}{a_{n}} \big(C_{+} 1_{\{ Z_{s-j}>0 \}} + C_{-} 1_{\{ Z_{s-j}<0 \}} \big)\\[0.4em]
    &\leq& \frac{\delta}{4(q+1) }\,(q+1) = \frac{\delta}{4}.
\end{eqnarray}
Since the event $H_{n, 1}^{c}$ occurs, for every $s \in \{1, \ldots, q\}$ we also have
\begin{equation}\label{e:phipm2}
  \frac{|X_{s}|}{a_{n}}  \leq  \sum_{j=0}^{q}|C_{j}| \frac{|Z_{s-j}|}{a_{n}} \leq \sum_{j=0}^{q}  \frac{C |Z_{s-j}|}{a_{n}} \leq (q+1) \frac{\delta}{4(q+1)} = \frac{\delta}{4}.
\end{equation}
Combining (\ref{e:phipm}) and (\ref{e:phipm2}) we obtain
\begin{equation}\label{e:phipm3}
 -\frac{\delta}{4} \leq \frac{X_{1}}{a_{n}} \leq M_{n} \Big( \frac{i^{*}}{n} \Big) = \bigvee_{s=1}^{i^{*}}\frac{X_{s}}{a_{n}}  \leq \frac{\delta}{4},
 \end{equation}
and hence
$$ \Big| W_{n} \Big( \frac{i^{*}}{n} \Big) - M_{n} \Big( \frac{i^{*}}{n} \Big) \Big| \leq  \Big| W_{n} \Big( \frac{i^{*}}{n} \Big) \Big| + \Big| M_{n} \Big( \frac{i^{*}}{n} \Big)\Big| \leq \frac{\delta}{4(q+1)} + \frac{\delta}{4} \leq \frac{\delta}{2},$$
which is in contradiction with (\ref{e:i*}).

Therefore $W_{n}(i^{*}/n) > \delta / [4(q+1)]$.
This implies the existence of some $k \in \{1,\ldots,i^{*}\}$ such that
\begin{equation}\label{e:est31}
 W_{n} \Big( \frac{i^{*}}{n} \Big) = \frac{ |Z_{k}|}{a_{n}} \big(C_{+} 1_{\{ Z_{k}>0 \}} + C_{-} 1_{\{ Z_{k}<0 \}} \big) > \frac{\delta}{4(q+1)},
\end{equation}
and hence
$$ \frac{C|Z_{k}|}{a_{n}} >  \frac{\delta}{4(q+1)}.$$
From this, since $H_{n, 1}^{c}$ occurs, it follows that $q+1 \leq k \leq n-q$. Since $H_{n, 2}^{c}$ occurs, it holds that
\begin{equation}\label{e:est51}
\frac{ C|Z_{l}|}{a_{n}}  \leq \frac{\delta}{4(q+1)} \qquad \textrm{for all} \ l \in \{k-q,\ldots,k+q\} \setminus \{k\}.
\end{equation}

Now we want to show that $M_{n}(i^{*}/n) = X_{j}/a_{n}$ for some $j \in \{1,\ldots,i^{*}\} \setminus \{k,\ldots,k+q\}$. If this is not the case, then $M_{n}(i^{*}/n) = X_{j}/a_{n}$ for some $j \in \{k,\ldots,k+q\}$ (with $j \leq i^{*}$). On the event $\{ Z_{k}>0 \}$ it holds that
 $$ |Z_{k}| \big(C_{+} 1_{\{ Z_{k}>0 \}} + C_{-} 1_{\{ Z_{k}<0 \}} \big) = C_{+}Z_{k} = C_{j_{0}}Z_{k}$$
 for some $j_{0} \in \{0,\ldots, q\}$ (with $C_{j_{0}} \geq 0$).
Here we distinguish two cases:
\begin{itemize}
\item[(i)] $k+q \leq i^{*}$.
Since $k+j_{0} \leq i^{*}$, we have
\begin{equation}\label{e:est6}
\frac{X_{j}}{a_{n}} = M_{n}  \Big( \frac{i^{*}}{n} \Big) \geq \frac{X_{k+j_{0}}}{a_{n}}.
\end{equation}
Observe that we can write
$$ \frac{X_{j}}{a_{n}} = \frac{C_{j-k}Z_{k}}{a_{n}} + \sum_{\scriptsize \begin{array}{c}
                          s=0  \\[-0.1em]
                          s \neq j-k
                        \end{array}}^{q} \frac{C_{s}Z_{j-s}}{a_{n}} =: \frac{C_{j-k}Z_{k}}{a_{n}} + F_{1},$$
and
$$ \frac{X_{k+j_{0}}}{a_{n}} = \frac{C_{j_{0}}Z_{k}}{a_{n}} + \sum_{\scriptsize \begin{array}{c}
                          s=0 \\[-0.1em]
                          s \neq j_{0}
                        \end{array}}^{q}\frac{C_{s}Z_{k+j_{0}-s}}{a_{n}} =: \frac{C_{j_{0}}Z_{k}}{a_{n}} + F_{2}.$$
From relation (\ref{e:est51}) (similarly as in (\ref{e:phipm2})) we obtain
$$ |F_{1}| \leq  q \cdot \frac{\delta}{4(q+1)} < \frac{\delta}{4},$$
and similarly $|F_{2}| < \delta/4$. Since
$ C_{j_{0}} - C_{j-k} = C_{+} - C_{j-k} \geq 0$,
 using (\ref{e:est6}) we obtain
$$ 0 \leq  \frac{C_{j_{0}} Z_{k} - C_{j-k}Z_{k}}{a_{n}} \leq F_{1}-F_{2} \leq |F_{1}| + |F_{2}| < \frac{\delta}{2}.$$
By (\ref{e:i*}) we have
$$ \Big| \frac{C_{j_{0}} Z_{k}}{a_{n}} - \frac{X_{j}}{a_{n}} \Big| = \Big| W_{n} \Big( \frac{i^{*}}{n} \Big) - M_{n} \Big( \frac{i^{*}}{n} \Big) \Big| > \delta,$$
and hence
$$ \delta < \Big| \frac{C_{j_{0}} Z_{k}}{a_{n}} -  \frac{C_{j-k}Z_{k}}{a_{n}} - F_{1} \Big| \leq \Big| \frac{C_{j_{0}} Z_{k}}{a_{n}} -  \frac{C_{j-k}Z_{k}}{a_{n}} \Big| + |F_{1}| < \frac{\delta}{2} + \frac{\delta}{4} = \frac{3\delta}{4},$$
which is not possible.
\item[(ii)] $k+q > i^{*}$. Note that in this case $k \leq j \leq i^{*} < k+q$. Since
 $$ M_{n} \Big( \frac{k+q}{n} \Big) = \bigvee_{s=1}^{k+q}\frac{X_{s}}{a_{n}} \geq M_{n}  \Big( \frac{i^{*}}{n} \Big) = \frac{X_{j}}{a_{n}},$$
 it holds that
 $$ M_{n} \Big( \frac{k+q}{n} \Big) = \frac{X_{p}}{a_{n}},$$
 for some $p \in \{j, \ldots, k+q \} \subseteq \{k, \ldots, k+q\}$.
 Observe that we can write
$$ \frac{X_{p}}{a_{n}} = \frac{C_{p-k}Z_{k}}{a_{n}} + \sum_{\scriptsize \begin{array}{c}
                          s=0  \\[-0.1em]
                          s \neq p-k
                        \end{array}}^{q} \frac{C_{s}Z_{p-s}}{a_{n}} =: \frac{C_{p-k}Z_{k}}{a_{n}} + F_{3},$$
with $|F_{3}| < \delta/4$, which holds by relation (\ref{e:est51}). By relation (\ref{e:i*q}) we have
$$ \Big| \frac{C_{j_{0}} Z_{k}}{a_{n}} - \frac{X_{p}}{a_{n}} \Big| = \Big| W_{n} \Big( \frac{i^{*}}{n} \Big) - M_{n} \Big( \frac{k+q}{n} \Big) \Big| > \delta,$$
and repeating the arguments as in (i), but with
$$ \frac{X_{p}}{a_{n}} = M_{n}  \Big( \frac{k+q}{n} \Big) \geq \frac{X_{k+j_{0}}}{a_{n}}$$
instead of (\ref{e:est6}), we arrive at
$$ \delta < \Big| \frac{C_{j_{0}} Z_{k}}{a_{n}} -  \frac{C_{p-k}Z_{k}}{a_{n}} - F_{3} \Big| \leq \Big| \frac{C_{j_{0}} Z_{k}}{a_{n}} -  \frac{C_{p-k}Z_{k}}{a_{n}} \Big| + |F_{3}| <  \frac{3\delta}{4}.$$
Thus we conclude that this case also can not happen.
\end{itemize}
One can similarly handle the event $\{ Z_{k} <0 \}$ to arrive at a contradiction. Therefore indeed $M_{n}(i^{*}/n) = X_{j}/a_{n}$ for some $j \in \{1,\ldots,i^{*}\} \setminus \{k,\ldots,k+q\}$.
Now we have three cases: (A1) all random variables $Z_{j-q}, \ldots, Z_{j}$ are "small", (A2) exactly one is "large" and (A3) at least two of them are "large", where we say $Z$ is "small" if $ C|Z| /a_{n} \leq \delta/[4(q+1)]$, otherwise it is "large". We will show that the first two cases are not possible.
\begin{itemize}
  \item[(A1)] $ C|Z_{j-l}|/a_{n} \leq \delta/[4(q+1)]$ for every $l=0,\ldots,q$.
 This yields (as in (\ref{e:phipm2}))
$$ \Big| M_{n} \Big( \frac{i^{*}}{n} \Big) \Big| = \frac{|X_{j}|}{a_{n}} \leq  \frac{\delta}{4}.$$
Let $j_{0}$ be as above (on the event $\{ Z_{k}>0 \}$), i.e.
 $$ |Z_{k}| \big(C_{+} 1_{\{ Z_{k}>0 \}} + C_{-} 1_{\{ Z_{k}<0 \}} \big) = C_{+}Z_{k} = C_{j_{0}}Z_{k}.$$
 If $k+q \leq i^{*}$, then
\begin{equation}\label{e:A1case}
 \frac{X_{j}}{a_{n}} \geq \frac{X_{k+j_{0}}}{a_{n}} = \frac{C_{j_{0}}Z_{k}}{a_{n}} + F_{2},
\end{equation}
where $F_{2}$ is as in (i) above, with $|F_{2}| < \delta/4$, i.e.
$$ F_{2} = \frac{X_{k+j_{0}}}{a_{n}} - \frac{C_{j_{0}}Z_{k}}{a_{n}} = \sum_{\scriptsize \begin{array}{c}
                          s=0 \\[-0.1em]
                          s \neq j_{0}
                        \end{array}}^{q}\frac{C_{s}Z_{k+j_{0}-s}}{a_{n}}.$$
 Hence
$$  \frac{C_{j_{0}}Z_{k}}{a_{n}} \leq \frac{X_{j}}{a_{n}} - F_{2} \leq \frac{|X_{j}|}{a_{n}} + |F_{2}| < \frac{\delta}{4} + \frac{\delta}{4} = \frac{\delta}{2},$$
and
$$ \Big| W_{n} \Big( \frac{i^{*}}{n} \Big) - M_{n} \Big( \frac{i^{*}}{n} \Big) \Big| = \Big| \frac{C_{j_{0}} Z_{k}}{a_{n}} - \frac{X_{j}}{a_{n}} \Big| \leq \frac{C_{j_{0}} Z_{k}}{a_{n}} + \frac{|X_{j}|}{a_{n}} < \frac{\delta}{2} + \frac{\delta}{4} = \frac{3\delta}{4},$$
 which is in contradiction with (\ref{e:i*}). On the other hand, if $k+q > i^{*}$, we have two possibilities: $M_{n}((k+q)/n) = M_{n}(i^{*}/n)$ or $M_{n}((k+q)/n) > M_{n}(i^{*}/n)$. When $M_{n}((k+q)/n) = M_{n}(i^{*}/n) = X_{j}/a_{n}$, since $k+j_{0} \leq k+q$ note that relation (\ref{e:A1case}) holds, and similarly as above we obtain
 $$ \Big| W_{n} \Big( \frac{i^{*}}{n} \Big) - M_{n} \Big( \frac{k+q}{n} \Big) \Big| = \Big| \frac{C_{j_{0}} Z_{k}}{a_{n}} - \frac{X_{j}}{a_{n}} \Big| <  \frac{3\delta}{4},$$
  which is in contradiction with (\ref{e:i*q}). Alternatively, when $M_{n}((k+q)/n) > M_{n}(i^{*}/n)$, it holds that  $M_{n}((k+q)/n) = X_{p}/a_{n}$ for some $p \in \{i^{*}, \ldots, k+q\}$. Now in the same manner as in (ii) above we get a contradiction. We handle the event $\{ Z_{k} <0 \}$ similarly to arrive at a contradiction, and therefore this case can not happen.

  \item[(A2)]  There exists $l_{1} \in \{0,\ldots,q\}$ such that $ C|Z_{j-l_{1}}|/a_{n} > \delta/[4(q+1)]$ and  $ C|Z_{j-l}|/a_{n} \leq \delta/[4(q+1)]$ for every $l \in \{0,\ldots,q\} \setminus \{l_{1}\}$. Here we analyze only what happens on the event $\{Z_{k}>0\}$ (the event $\{Z_{k}<0\}$ can be treated analogously and is therefore omitted). Assume first $k+q \leq i^{*}$. Then
\begin{equation}\label{e:est7}
\frac{X_{j}}{a_{n}} \geq \frac{X_{k+j_{0}}}{a_{n}} = \frac{C_{j_{0}}Z_{k}}{a_{n}} + F_{2},
\end{equation}
where $j_{0}$ and $F_{2}$ are as in (i) above, with $|F_{2}| < \delta/4$.
 Write
$$ \frac{X_{j}}{a_{n}} = \frac{C_{l_{1}}Z_{j-l_{1}}}{a_{n}} + \sum_{\scriptsize \begin{array}{c}
                          s=0  \\[-0.1em]
                          s \neq l_{1}
                        \end{array}}^{q} \frac{C_{s}Z_{j-s}}{a_{n}} =: \frac{C_{l_{1}}Z_{j-l_{1}}}{a_{n}} + F_{4}.$$
Similarly as before we obtain $|F_{4}| < \delta/4$. Since $j-l_{1} \leq j \leq i^{*}$, by the definition of the process $W_{n}(\,\cdot\,)$ we have
$$ W_{n} \Big( \frac{i^{*}}{n} \Big) \geq \frac{|Z_{j-l_{1}}|}{a_{n}} \big(C_{+} 1_{\{ Z_{j-l_{1}}>0 \}} + C_{-} 1_{\{ Z_{j-l_{1}}<0 \}} \big) \geq \frac{C_{l_{1}}Z_{j-l_{1}}}{a_{n}}.$$
Thus
$$ \frac{C_{j_{0}} Z_{k}}{a_{n}} = \frac{|Z_{k}|}{a_{n}} \big(C_{+} 1_{\{ Z_{k}>0 \}} + C_{-} 1_{\{ Z_{k}<0 \}} \big) = W_{n} \Big( \frac{i^{*}}{n} \Big) \geq \frac{C_{l_{1}}Z_{j-l_{1}}}{a_{n}}, $$
which yields
\begin{equation}\label{e:est8}
\frac{ C_{j_{0}} Z_{k}}{a_{n}} - \frac{X_{j}}{a_{n}} \geq \frac{C_{l_{1}}Z_{j-l_{1}}}{a_{n}} - \frac{X_{j}}{a_{n}} = - F_{4}.
\end{equation}
Relations (\ref{e:est7}) and (\ref{e:est8}) yield
$$- (|F_{2}| + |F_{4}|) \leq -F_{4} \leq \frac{C_{j_{0}} Z_{k}}{a_{n}} - \frac{X_{j}}{a_{n}} \leq -F_{2} \leq  |F_{2}| + |F_{4}|,$$
i.e.
$$ \Big| W_{n} \Big( \frac{i^{*}}{n} \Big) - M_{n} \Big( \frac{i^{*}}{n} \Big) \Big| = \Big| \frac{C_{j_{0}} Z_{k}}{a_{n}} - \frac{X_{j}}{a_{n}} \Big| \leq |F_{2}| + |F_{4}| < \frac{\delta}{4} + \frac{\delta}{4}=\frac{\delta}{2},$$
which is in contradiction with (\ref{e:i*}).
 Assume now $k+q > i^{*}$. If $M_{n}((k+q)/n) = M_{n}(i^{*}/n) = X_{j}/a_{n}$, relation (\ref{e:est7}) still holds and this leads to
 $$ \Big| W_{n} \Big( \frac{i^{*}}{n} \Big) - M_{n} \Big( \frac{k+q}{n} \Big) \Big| = \Big| \frac{C_{j_{0}} Z_{k}}{a_{n}} - \frac{X_{j}}{a_{n}} \Big| <  \frac{\delta}{2},$$
  which is in contradiction with (\ref{e:i*q}). On the other hand, if $M_{n}((k+q)/n) > M_{n}(i^{*}/n)$, then $M_{n}((k+q)/n) = X_{p}/a_{n}$ for some $p \in \{i^{*}, \ldots, k+q\}$. With the same arguments as in (ii) above
 we obtain $$ \Big| W_{n} \Big( \frac{i^{*}}{n} \Big) - M_{n} \Big( \frac{k+q}{n} \Big) \Big| = \Big| \frac{C_{j_{0}} Z_{k}}{a_{n}} - \frac{X_{p}}{a_{n}} \Big| <  \frac{3\delta}{4},$$
 i.e.~a contradiction with (\ref{e:i*q}). Hence this case also can not happen.

  \item[(A3)] There exist $l_{1} \in \{0,\ldots,q\}$ and $l \in \{0,\ldots,q\} \setminus \{l_{1}\}$ such that $ C|Z_{j-l_{1}}|/a_{n} > \delta/[4(q+1)]$ and $ CZ_{j-l}|/a_{n} > \delta/[4(q+1)]$. In this case the event $H_{n, 3}$ occurs.
\end{itemize}

Therefore only case (A3) is possible, and this yields $\widehat{D}_{n} \subseteq H_{n, 3}$. Hence (\ref{e:estim1}) holds.
By stationarity we have
\begin{equation}\label{e:Hn1new}
 \Pr(H_{n, 1}) \leq (3q+1) \Pr \bigg( \frac{C|Z_{1}|}{a_{n}} > \frac{\delta}{4(q+1)} \bigg).
\end{equation}
For an arbitrary $M>0$ it holds that
   \begin{eqnarray*}
     \Pr \bigg( \frac{C|Z_{1}|}{a_{n}} > \frac{\delta}{4(q+1)} \bigg) & = &
      \Pr \bigg( \frac{C|Z_{1}|}{a_{n}} > \frac{\delta}{4(q+1)},\,C > M \bigg) + \Pr \bigg( \frac{C|Z_{1}|}{a_{n}} > \frac{\delta}{4(q+1)},\,C \leq M \bigg)\\[0.6em]
       &\leq &  \Pr(C > M) + \Pr \bigg( \frac{|Z_{1}|}{a_{n}} > \frac{\delta}{4(q+1)M} \bigg).
   \end{eqnarray*}
Using the regular variation property
we obtain
\begin{equation*}
\lim_{n \to \infty}  \Pr \bigg( \frac{|Z_{1}|}{a_{n}} > \frac{\delta}{4(q+1)M} \bigg) =0,
\end{equation*}
and therefore from (\ref{e:Hn1new}) we get
$ \limsup_{n \to \infty} \Pr (H_{n,1}) \leq (3q+1) \Pr (C > M).$
Letting $M \to \infty$ we conclude
\begin{equation}\label{e:est2}
\lim_{n \to \infty} \Pr(H_{n, 1})=0.
\end{equation}
Since $Z_{k}$ and $Z_{l}$ that appear in the formulation of $H_{n, 2}$ are independent, for an arbitrary $M>0$ it holds that
\begin{eqnarray}
 \nonumber \Pr(H_{n, 2} \cap \{C \leq M\}) &=& \sum_{k=1}^{n} \sum_{\scriptsize \begin{array}{c}
                          l=k-q  \\[-0.1em]
                          l \neq k
                        \end{array}}^{k+q} \Pr \bigg( \frac{C|Z_{k}|}{a_{n}} > \frac{\delta}{4(q+1)},\,\frac{C|Z_{l}|}{a_{n}} > \frac{\delta}{4(q+1)},\,C \leq M \bigg)\\[-0.1em]
   \nonumber &  \leq & \sum_{k=1}^{n} \sum_{\scriptsize \begin{array}{c}
                          l=k-q  \\[-0.1em]
                          l \neq k
                        \end{array}}^{k+q} \Pr \bigg( \frac{|Z_{k}|}{a_{n}} > \frac{\delta}{4(q+1)M} \bigg) \Pr \bigg( \frac{|Z_{l}|}{a_{n}} > \frac{\delta}{4(q+1)M}\bigg)\\[-0.1em]
  \nonumber & = & \frac{2q}{n} \bigg[ n \Pr \bigg( \frac{|Z_{1}|}{a_{n}} > \frac{\delta}{4(q+1)M} \bigg) \bigg]^{2},
 \end{eqnarray}
and an application of the regular variation property yields
$ \lim_{n \to \infty} \Pr(H_{n, 2} \cap \{ C \leq M \})=0$.
Hence
 $$ \limsup_{n \to \infty}  \Pr(H_{n, 2}) \leq \limsup_{n \to \infty} \Pr(H_{n, 2} \cap \{C > M\}) \leq \Pr (C > M),$$
and letting again $M \to \infty$ we conclude
\begin{equation}\label{e:est5}
\lim_{n \to \infty} \Pr(H_{n, 2})=0.
\end{equation}
From the definition of the set $H_{n, 3}$ it follows that $k, j-l_{1}, j-l$ are all different, which implies that the random variables $Z_{k}$, $Z_{j-l_{1}}$ and $Z_{j-l}$ are independent. Using this and stationarity we obtain
$$ \Pr(H_{n, 3} \cap \{C \leq M\}) \leq \frac{q(q+1)}{n} \bigg[ n \Pr \bigg( \frac{|Z_{1}|}{a_{n}} > \frac{\delta}{4(q+1)M} \bigg) \bigg]^{3} $$
for arbitrary $M>0$,
and hence
\begin{equation}\label{e:est10}
\lim_{n \to \infty} \Pr(H_{n, 3})=0.
\end{equation}
Now from (\ref{e:estim1}) and (\ref{e:est2})--(\ref{e:est10}) we obtain
$ \lim_{n \to \infty} \Pr(D_{n})=0,$
and hence (\ref{e:Yn}) yields
\begin{equation}\label{eq:Ynend}
\lim_{n \to \infty} \Pr(Y_{n}> \delta)=0.
\end{equation}
It remains to estimate the second term on the right hand side of (\ref{eq:AB}). Let
$$ E_{n} = \{\exists\,v \in \Gamma_{M_{n}} \ \textrm{such that} \ d(v,z) > \delta \ \textrm{for every} \ z \in \Gamma_{ W_{n}} \}.$$
Then by the definition of $T_{n}$
\begin{equation}\label{e:Tnfirst}
 \{T_{n} > \delta\} \subseteq E_{n}.
\end{equation}
On the event $E_{n}$ it holds that $d(v, \Gamma_{ W_{n}})> \delta$.
Interchanging the roles of the processes $M_{n}(\,\cdot\,)$ and $W_{n}(\,\cdot\,)$, in the same way as before for the event $D_{n}$ it can be shown that
\begin{equation}\label{e:i*qneg}
 \Big| W_{n} \Big( \frac{i^{*}-l}{n} \Big) - M_{n} \Big( \frac{i^{*}}{n} \Big) \Big| > \delta
\end{equation}
for all $l=0,\ldots, q$ (such that $i^{*}-l \geq 0$), where $i^{*}=\floor{nt_{v}}$ or $i^{*}=\floor{nt_{v}}-1$, and $v=(t_{v},x_{v})$.

Now we want to show that $E_{n} \cap (H_{n, 1} \cup H_{n, 2})^{c} \subseteq H_{n, 3}$, and hence assume the event $E_{n} \cap (H_{n, 1} \cup H_{n, 2})^{c}$ occurs. Since (\ref{e:i*qneg}) (for $l=0$) is in fact (\ref{e:i*}), repeating the arguments used for $D_{n}$ we conclude that (\ref{e:est31}) holds. Here we also claim that $M_{n}(i^{*}/n) = X_{j}/a_{n}$ for some $j \in \{1,\ldots,i^{*}\} \setminus \{k,\ldots,k+q\}$. Hence assume this is not the case, i.e. $M_{n}(i^{*}/n) = X_{j}/a_{n}$ for some $j \in \{k,\ldots,k+q\}$ (with $j \leq i^{*}$). We can repeat the arguments from (i) above to conclude that $k + q \leq i^{*}$ is not possible. It remains to see what happens when $k + q > i^{*}$. Let
$$ W_{n} \Big( \frac{i^{*}-q}{n} \Big) = \frac{ |Z_{s}|}{a_{n}} \big(C_{+} 1_{\{ Z_{s}>0 \}} + C_{-} 1_{\{ Z_{s}<0 \}} \big)$$
for some $s \in \{1, \ldots, i^{*}-q\}$. Note that $i^{*}-q \geq 1$ since $q+1 \leq k  \leq i^{*}$. We distinguish two cases:
\begin{itemize}
\item[(a)] $W_{n}(i^{*}/n) > M_{n}(i^{*}/n)$. In this case the definition of $i^{*}$ implies that $M_{n}(i^{*}/n) \leq x_{v} \leq W_{n}(i^{*}/n)$. Since $|t_{v}-(i^{*}-q)/n| < (q+1)/n \leq \delta$,  from $d(v, \Gamma_{W_{n}})> \delta$ we conclude
$$ \widetilde{d} \Big((x_{v}, \Big[ W_{n} \Big( \frac{i^{*}-q}{n} \Big), W_{n} \Big(\frac{i^{*}}{n} \Big) \Big] \Big) > \delta,$$
where $\widetilde{d}$ is the Euclidean metric on $\mathbb{R}$. This yields
$W_{n} ((i^{*}-q)/n) > M_{n} (i^{*}/n)$,
and from (\ref{e:i*qneg}) we obtain
\begin{equation}\label{e:estE1}
 W_{n} \Big( \frac{i^{*}-q}{n} \Big) > M_{n} \Big( \frac{i^{*}}{n} \Big) + \delta.
\end{equation}
From this, taking into account relation (\ref{e:phipm3}), we obtain
$$ \frac{C|Z_{s}|}{a_{n}} \geq   \,W_{n} \Big( \frac{i^{*}-q}{n} \Big) >  -\frac{\delta}{4} + \delta = \frac{3\delta}{4} > \frac{\delta}{4(q+1)},$$
and since $H_{n, 2}^{c}$ occurs it follows that
\begin{equation}\label{e:estE12}
\frac{C|Z_{l}|}{a_{n}} \leq \frac{\delta}{4(q+1)} \quad \textrm{for every} \  l \in \{s-q,\ldots,s+q\} \setminus \{s\}.
\end{equation}
Let $p_{0} \in \{0,\ldots,q\}$ be such that $C_{p_{0}}Z_{s} =  |Z_{s}| \big(C_{+} 1_{\{ Z_{s}>0 \}} + C_{-} 1_{\{ Z_{s}<0 \}} \big) $.
Since $s + p_{0} \leq i^{*}$, it holds that
\begin{equation}\label{e:estE2}
 \frac{X_{j}}{a_{n}} = M_{n} \Big( \frac{i^{*}}{n} \Big) \geq \frac{X_{s+p_{0}}}{a_{n}} = \frac{C_{p_{0}}Z_{s}}{a_{n}} + F_{5},
\end{equation}
where
$$ F_{5} =  \sum_{\scriptsize \begin{array}{c}
                          m=0 \\[-0.1em]
                          m \neq p_{0}
                        \end{array}}^{q}\frac{C_{m}Z_{s+p_{0}-m}}{a_{n}}.$$
From (\ref{e:estE1}) and (\ref{e:estE2}) we obtain
$$ \frac{C_{p_{0}} Z_{s}}{a_{n}} > \frac{X_{j}}{a_{n}} + \delta \geq \frac{C_{p_{0}} Z_{s}}{a_{n}} + F_{5} + \delta,$$
i.e.
$F_{5} < -\delta$. But this is not possible since by (\ref{e:estE12}), $|F_{5}| \leq \delta/4$,
and we conclude that this case can not happen.
\item[(b)] $W_{n}(i^{*}/n) \leq M_{n}(i^{*}/n)$. Then from (\ref{e:i*qneg}) we get
\begin{equation}\label{e:estE3}
 M_{n} \Big( \frac{k+q}{n} \Big) \geq M_{n} \Big( \frac{i^{*}}{n} \Big) \geq W_{n} \Big( \frac{i^{*}}{n} \Big) + \delta.
\end{equation}
Therefore
$$ \Big| W_{n} \Big( \frac{i^{*}}{n} \Big) - M_{n} \Big(\frac{k+q}{n} \Big) \Big| > \delta,$$
and repeating the arguments from (ii) above we conclude that this case also can not happen.
\end{itemize}
Thus we have proved that  $M_{n}(i^{*}/n) = X_{j}/a_{n}$ for some $j \in \{1,\ldots,i^{*}\} \setminus \{k,\ldots,k+q\}$. Similar as before one can prove now that Cases (A1) and (A2) can not happen (when $k+q > i^{*}$ we use also the arguments from (a) and (b)), which means that only Case (A3) is possible. In that case the event $H_{n, 3}$ occurs, and thus we have proved that $E_{n} \cap (H_{n, 1} \cup H_{n, 2} \cup)^{c} \subseteq H_{n, 3}$. Hence
$$ E_{n} \subseteq H_{n, 1} \cup H_{n, 2} \cup H_{n, 3},$$
and from (\ref{e:est2})--(\ref{e:est10}) we obtain
$ \lim_{n \to \infty} \Pr(E_{n})=0.$
Therefore (\ref{e:Tnfirst}) yields
\be\label{eq:Tnend}
\lim_{n \to \infty} \Pr(T_{n}> \delta)=0.
\ee
Now from (\ref{eq:AB}), (\ref{eq:Ynend}) and (\ref{eq:Tnend}) we obtain (\ref{e:max1}),
which means that $M_{n}(\,\cdot\,) \dto M(\,\cdot\,)$
in $D_{\uparrow}[0,1]$ with the $M_{1}$ topology. This proves the theorem.
\end{proof}

\begin{rem}
If the sequence of coefficients $(C_{j})$ is a.s.~of the same sign, the limiting process $M(\cdot)$ in Theorem~\ref{t:FLT} reduces to
$C^{(*)}W^{(*)}(\cdot)$ with $W^{(*)}$ being an extremal process with exponent measure $\mu_{*}$, independent of $C^{(*)}$, where in the case of non-negative coefficients $C^{(*)} = C^{(1)} \eind  \max \{ C_{j} \vee 0 : j=0, \ldots, q\}$ and $\mu_{*}(\rmd x) = \mu_{+}(\rmd x) = p \alpha x^{-\alpha-1} \rmd x$ for $x>0$, while in the case of non-positive coefficients $C^{(*)} =  C^{(2)} \eind  \max \{ -C_{j} \vee 0 : j=0, \ldots, q\}$ and $\mu_{*}(\rmd x) = \mu_{-}(\rmd x) = r \alpha x^{-\alpha-1} \rmd x$ for $x>0$. Similarly, if the innovations $(Z_{i})$ are a.s.~of the same sign, the limiting process is again $C^{(*)}W^{(*)}(\cdot)$, with $\mu_{*}(\rmd x) = \alpha x^{-\alpha-1} \rmd x$ for $x>0$, and $C^{(*)}=C^{(1)}$ if the innovations are non-negative and $C^{(*)}=C^{(2)}$ if they are non-positive.
\end{rem}

Now we turn our attention to infinite order linear processes.
 The idea is to approximate them by a sequence of finite order linear processes, for which Theorem~\ref{t:FLT} holds, and to show that the error of approximation is negligible in the limit. To accomplish this in the case $\alpha \in (0,1)$ we will use the arguments from Krizmani\'{c}~\cite{Kr19}, and in the case $\alpha \in [1,\infty)$ the arguments from the proof of Lemma 2 in Tyran-Kami\'{n}ska~\cite{Ty10b} adapted to linear processes with random coefficients instead of deterministic.

\begin{thm}\label{t:infFLT}
Let $(X_{i})$ be a linear process defined by
$$ X_{i} = \sum_{j=0}^{\infty}C_{j}Z_{i-j}, \qquad i \in \mathbb{Z},$$
where $(Z_{i})_{i \in \mathbb{Z}}$ is an i.i.d.~sequence of random variables satisfying $(\ref{e:regvar})$ and $(\ref{e:pr})$ with $\alpha >0$, and
 $(C_{i})_{i \geq 0 }$ is a sequence of random variables independent of $(Z_{i})$ such that the series defying the above linear process is a.s.~convergent. If $\alpha \in (0,1)$ suppose
 \begin{equation}\label{e:momcondr}
 \sum_{j=0}^{\infty} \mathrm{E} |C_{j}|^{\delta} < \infty \qquad \textrm{for some}  \ \delta \in (0, \alpha),
 \end{equation}
 and
\begin{equation}\label{e:mod1}
  \sum_{j=0}^{\infty}\mathrm{E}|C_{j}|^{\gamma} < \infty \qquad \textrm{for some} \ \gamma \in (\alpha, 1),
\end{equation}
 while if $\alpha = 1$ suppose also $(\ref{e:momcondr})$, and if $\alpha >1$ suppose
\begin{equation}\label{eq:infmaTK3}
\sum_{j=0}^{\infty} \mathrm{E}|C_{j}| < \infty.
\end{equation}
 Then
$M_{n}(\,\cdot\,) \dto  M(\,\cdot\,)$, as $n \to \infty$,
in $D_{\uparrow}[0,1]$ with the $M_{1}$ topology, with $M$ as defined in $(\ref{e:limprocess})$.
\end{thm}

\begin{proof}
For $q \in \mathbb{N}$, $q \geq 2$, define
$$ X_{i}^{q} = \sum_{j=0}^{q-2}C_{j}Z_{i-j} + C^{q, \textrm{max}} Z_{i-q+1} + C^{q, \textrm{min}} Z_{i-q}, \qquad i \in \mathbb{Z},$$
and
$$ M_{n, q}(t) = \bigvee_{i=1}^{\floor{nt}} \frac{X_{i}^{q}}{a_{n}}, \qquad t \in [0,1]$$
(with $M_{n, q}(t)=X_{1}^{q}/a_{n}$ if $t \in [0, 1/n)$), where $C^{q, \textrm{max}}= \max \{C_{j} : j \geq q-1\}$ and $C^{q, \textrm{min}}= \min \{C_{j} : j \geq q-1\}$.
Observe that
$$\max \{ C_{j} \vee 0 : j =0,\ldots, q-1\} \vee (C^{q, \textrm{max}} \vee 0) \vee (C^{q, \textrm{min}} \vee 0) = C_{+},$$
 $$ \max \{ -C_{j} \vee 0 : j =0,\ldots, q-1\} \vee (-C^{q, \textrm{max}} \vee 0) \vee (-C^{q, \textrm{min}} \vee 0) = C_{-},$$
and therefore for the finite order moving average process $(X_{i}^{q})_{i}$ by Theorem~\ref{t:FLT} we obtain
$$ M_{n, q}(\,\cdot\,)  \dto M(\,\cdot\,) \qquad \textrm{as} \ n \to \infty,$$
in $D_{\uparrow}[0,1]$ with the $M_{1}$ topology. If we show that for every $\epsilon >0$
\begin{equation}\label{e:Slutskyinf01}
 \lim_{q \to \infty} \limsup_{n \to \infty}\Pr[d_{M_{1}}(M_{n}, M_{n, q})> \epsilon]=0,
\end{equation}
then by a generalization of Slutsky's theorem (see Theorem 3.5 in Resnick~\cite{Resnick07}) it will follow $M_{n}(\,\cdot\,) \dto M(\,\cdot\,)$ in $D_{\uparrow}[0,1]$ with the $M_{1}$ topology. Since the metric $d_{M_{1}}$ on $D_{\uparrow}[0,1]$ is bounded above by the uniform metric on $D_{\uparrow}[0,1]$, it suffices to show that
$$ \lim_{q \to \infty} \limsup_{n \to \infty}\Pr \bigg( \sup_{0 \leq t \leq 1}|M_{n}(t) - M_{n, q}(t)|> \epsilon \bigg)=0.$$
Now we treat separately the cases $\alpha \in (0,1)$ and $\alpha \in [1,\infty)$.\\[-0.2em]

Case $\alpha \in (0,1)$.
Recalling the definitions, we have
\begin{equation*}
 \Pr \bigg( \sup_{0 \leq t \leq 1}|M_{n}(t) - M_{n, q}(t)|> \epsilon \bigg) \leq \Pr \bigg( \bigvee_{i=1}^{n}\frac{|X_{i}-X_{i}^{q}|}{a_{n}} > \epsilon \bigg) \leq \Pr \bigg( \sum_{i=1}^{n}\frac{|X_{i}-X_{i}^{q}|}{a_{n}} > \epsilon \bigg).
\end{equation*}
Observe that
\begin{eqnarray*}
  \sum_{i=1}^{n}|X_{i}-X_{i}^{q}| & = & \sum_{i=1}^{n} \bigg| \sum_{j=0}^{\infty}C_{j}Z_{i-j} - \sum_{j=0}^{q-2}C_{j}Z_{i-j} - C^{q, \textrm{max}}Z_{i-q+1} -  C^{q, \textrm{min}}Z_{i-q}\bigg| \\[0.1em]
  & \hspace*{-8em} \leq & \hspace*{-4em} \sum_{i=1}^{n} \bigg[ |(C_{q-1} - C^{q, \textrm{max}}) Z_{i-q+1}| + |(C_{q} - C^{q, \textrm{min}}) Z_{i-q}| + \sum_{j=q+1}^{\infty}|C_{j} Z_{i-j}| \bigg]\\[0.1em]
   & \hspace{-8em} \leq & \hspace*{-4em} \bigg( 2 \sum_{j=q+1}^{\infty}|C_{j}| \bigg) \sum_{i=1}^{n} ( |Z_{i-q+1}| + |Z_{i-q}|) + \sum_{i=1}^{n} \sum_{j=q+1}^{\infty}|C_{j}|\,|Z_{i-j}|\\[0.1em]
    & \hspace{-8em} \leq & \hspace*{-4em} \bigg( 5 \sum_{j=q+1}^{\infty}|C_{j}| \bigg) \sum_{i=1}^{n+1} |Z_{i-q}| + \sum_{i=-\infty}^{0}|Z_{i-q}| \sum_{j=1}^{n}|C_{q-i+j}|,
\end{eqnarray*}
where in the second inequality above we used the simple fact that $|C_{q-1} - C^{q, \textrm{max}}| \leq 2 \sum_{j=q+1}^{\infty}|C_{j}|$ (and analogously if $C^{q, \textrm{max}}$ is replaced by $C^{q, \textrm{min}}$), and a change of variables and rearrangement of sums in the third inequality. Since conditions $(\ref{e:momcondr})$ and $(\ref{e:mod1})$ by Lemma 3.2 in Krizmani\'{c}~\cite{Kr19} imply
$$ \lim_{q \to \infty} \limsup_{n \to \infty} \Pr \bigg[ \bigg( 5 \sum_{j=q+1}^{\infty}|C_{j}| \bigg) \sum_{i=1}^{n+1}\frac{|Z_{i-q}|}{a_{n}} + \sum_{i=-\infty}^{0}\frac{|Z_{i-q}|}{a_{n}} \sum_{j=1}^{n}|C_{q-i+j}| > \epsilon \bigg] =0,$$
we obtain
$$ \lim_{q \to \infty} \limsup_{n \to \infty}  \Pr \bigg( \sup_{0 \leq t \leq 1}|M_{n}(t) - M_{n, q}(t)|> \epsilon \bigg)= 0,$$
which means that $M_{n}(\,\cdot\,) \dto M(\,\cdot\,)$ as $n \to \infty$ in $(D_{\uparrow}[0,1], d_{M_{1}})$.\\[-0.4em]

Case $\alpha \in [1,\infty)$.
Define $Z_{n,j}^{\leq} = a_{n}^{-1} Z_{j} 1_{\{ |Z_{j}| \leq a_{n} \}}$ and $Z_{n,j}^{>} = a_{n}^{-1} Z_{j} 1_{\{ |Z_{j}| > a_{n} \}}$  for $j \in \mathbb{Z}$ and $n \in \mathbb{N}$,
$$ \widetilde{C}_{j} = \left\{ \begin{array}{cc}
                                   C_{j}, & \quad \textrm{if} \ j > q,\\[0.3em]
                                   C_{q}-C^{q, \textrm{min}}, & \quad \textrm{if} \ j=q,\\[0.3em]
                                   C_{q-1}-C^{q, \textrm{max}}, & \quad \textrm{if} \ j=q-1,
                                 \end{array}\right.$$
and note that
\begin{eqnarray*}
\nonumber  |M_{n}(t) - M_{n,q}(t)| & = & \bigg| \bigvee_{i=1}^{\floor{nt}} \frac{X_{i}}{a_{n}} -  \bigvee_{i=1}^{\floor{nt}} \frac{X_{i}^{q}}{a_{n}} \bigg| \leq  \bigvee_{i=1}^{\floor{nt}} \frac{|X_{i}-X_{i}^{q}|}{a_{n}} = \bigvee_{i=1}^{\floor{nt}} \bigg| \sum_{j=q-1}^{\infty} \frac{\widetilde{C}_{j}Z_{i-j}}{a_{n}} \bigg| \\[0.1em]
& = &  \bigvee_{i=1}^{\floor{nt}}  \bigg| \sum_{j=q-1}^{\infty} \widetilde{C}_{j}Z_{n,i-j}^{\leq} +  \sum_{j=q-1}^{\infty} \widetilde{C}_{j}Z_{n,i-j}^{>} \bigg|.
\end{eqnarray*}
Using again the fact that the $M_{1}$ metric on $D_{\uparrow}[0,1]$ is bounded above by the uniform metric we get
\begin{eqnarray}\label{eq:I1I2}
\nonumber   \Pr[d_{M_{1}}(M_{n}, M_{n, q})> \epsilon] & \leq & \Pr \bigg( \sup_{0 \leq t \leq 1} |M_{n}(t) - M_{n,q}(t)| > \epsilon \bigg)\\[0.4em]
 & \hspace*{-30em} \leq & \hspace*{-15em} \Pr \bigg(\bigvee_{i=1}^{n} \bigg| \sum_{j=q-1}^{\infty} \widetilde{C}_{j}Z_{n,i-j}^{\leq} \bigg| > \frac{\epsilon}{2} \bigg) + \Pr \bigg( \bigvee_{i=1}^{n} \bigg| \sum_{j=q-1}^{\infty} \widetilde{C}_{j}Z_{n,i-j}^{>} \bigg| > \frac{\epsilon}{2} \bigg) =: I_{1} + I_{2}.
\end{eqnarray}
To estimate $I_{1}$ note that
 \begin{eqnarray*}
 I_{1} & \leq & \Pr \bigg(\sum_{j=q-1}^{\infty}|\widetilde{C}_{j}| > 1 \bigg) + \Pr \bigg(\bigvee_{i=1}^{n} \bigg| \sum_{j=q-1}^{\infty} \widetilde{C}_{j}Z_{n,i-j}^{\leq} \bigg| > \frac{\epsilon}{2},\,\sum_{j=q-1}^{\infty}|\widetilde{C}_{j}| \leq 1 \bigg)\\[0.2em]
  & \leq & \sum_{j=q-1}^{\infty} \mathrm{E}|\widetilde{C}_{j}| + \sum_{i=1}^{n} \Pr \bigg(\bigg| \sum_{j=q-1}^{\infty} \widetilde{C}_{j}Z_{n,i-j}^{\leq} \bigg| > \frac{\epsilon}{2},\,\sum_{j=q-1}^{\infty}|\widetilde{C}_{j}| \leq 1 \bigg),
 \end{eqnarray*}
where the last inequality follows by Markov's inequality.
Take some $\varphi > \alpha$ and let $\psi$ be such that $1/\varphi + 1/\psi=1$. Then by H\"{o}lder's inequality we have
  \begin{eqnarray*}
 \bigg( \sum_{j=q-1}^{\infty} |\widetilde{C}_{j} Z_{n,i-j}^{\leq}| \bigg)^{\varphi}  &=&  \bigg( \sum_{j=q-1}^{\infty} |\widetilde{C}_{j}|^{1/\psi} \cdot |\widetilde{C}_{j}|^{1/\varphi} |Z_{n,i-j}^{\leq}| \bigg)^{\varphi}\\[0.4em]
   &\leq& \bigg( \sum_{j=q-1}^{\infty}|\widetilde{C}_{j}| \bigg)^{\varphi/\psi} \sum_{j=q-1}^{\infty}|\widetilde{C}_{j}| \cdot | Z_{n,i-j}^{\leq}|^{\varphi},
 \end{eqnarray*}
which leads to
\begin{eqnarray*}
 I_{1} & \leq &  \sum_{j=q-1}^{\infty} \mathrm{E}|\widetilde{C}_{j}| + \sum_{i=1}^{n} \Pr \bigg[ \bigg( \sum_{j=q-1}^{\infty}|\widetilde{C}_{j}| \bigg)^{\varphi/\psi} \sum_{j=q-1}^{\infty}|\widetilde{C}_{j}| \cdot | Z_{n,i-j}^{\leq}|^{\varphi} > \Big( \frac{\epsilon}{2} \Big)^{\varphi},\,\sum_{j=q-1}^{\infty}|\widetilde{C}_{j}| \leq 1 \bigg]\\[0.3em]
  &  \leq &  \sum_{j=q-1}^{\infty} \mathrm{E}|\widetilde{C}_{j}| + \sum_{i=1}^{n} \Pr \bigg[\sum_{j=q-1}^{\infty}|\widetilde{C}_{j}| \cdot | Z_{n,i-j}^{\leq}|^{\varphi} > \Big( \frac{\epsilon}{2} \Big)^{\varphi} \bigg].
\end{eqnarray*}
 This together with the Markov's inequality, the fact that the sequence $(C_{i})$ is independent of $(Z_{i})$ and stationarity of the sequence $(Z_{i})$ yields

 \begin{eqnarray}\label{e:I1new}
  \nonumber I_{1} & \leq & \sum_{j=q-1}^{\infty} \mathrm{E}|\widetilde{C}_{j}| +  \Big( \frac{\epsilon}{2} \Big)^{-\varphi} \sum_{i=1}^{n} \mathrm{E} \bigg( \sum_{j=q-1}^{\infty}|\widetilde{C}_{j}| \cdot | Z_{n,i-j}^{\leq}|^{\varphi} \bigg)\\[0.3em]
   \nonumber & = & \sum_{j=q-1}^{\infty} \mathrm{E}|\widetilde{C}_{j}| +  \frac{2^{\varphi}}{\epsilon^{\varphi}} \sum_{i=1}^{n} \sum_{j=q-1}^{\infty} \mathrm{E} |\widetilde{C}_{j}|\,\mathrm{E}|Z_{n,i-j}^{\leq}|^{\varphi}\\[0.3em]
   & = & \Big(1+  \frac{2^{\varphi}}{\epsilon^{\varphi}} n \mathrm{E}|Z_{n, 1}^{\leq}|^{\varphi} \Big) \sum_{j=q-1}^{\infty} \mathrm{E}|\widetilde{C}_{j}|.
  \end{eqnarray}
From the definition of $\widetilde{C}_{j}$ it follows $\sum_{j=q-1}^{\infty} \mathrm{E}|\widetilde{C}_{j}| \leq 5 \sum_{j=q-1}^{\infty} \mathrm{E}|C_{j}|$, and
by Karamata's theorem and (\ref{eq:niz}), as $n \to \infty$,
$$ n \mathrm{E}|Z_{n, 1}^{\leq}|^{\varphi}  = \frac{\mathrm{E} ( |Z_{1}|^{\varphi} 1 _{\{ |Z_{1}| \leq a_{n} \}})}{a_{n}^{\varphi} \Pr (|Z_{1}| > a_{n})} \cdot n \Pr(|Z_{1}| > a_{n}) \to \frac{\alpha}{\varphi-\alpha} < \infty.$$
Hence from (\ref{e:I1new}) we conclude that there exists a positive constant $D_{1}$ such that
\begin{equation}\label{e:I1new2}
 \limsup_{n \to \infty} I_{1} \leq D_{1} \sum_{j=q-1}^{\infty} \mathrm{E}|C_{j}|.
\end{equation}
In order to estimate $I_{2}$ assume first $\alpha \in (1,\infty)$. Applying again Markov's inequality, the fact that the sequence $(C_{i})$ is independent of $(Z_{i})$ and stationarity of $(Z_{i})$ we obtain
\begin{equation*}\label{eq:alpha1}
\nonumber I_{2}  \leq    \sum_{i=1}^{n} \Pr \bigg( \bigg| \sum_{j=q-1}^{\infty} \widetilde{C}_{j}Z_{n,i-j}^{>} \bigg| > \frac{\epsilon}{2} \bigg) \leq \frac{2}{\epsilon} \sum_{i=1}^{n}  \sum_{j=q-1}^{\infty} \mathrm{E} |\widetilde{C}_{j}Z_{n,i-j}^{>}| = \frac{2}{\epsilon}  \sum_{j=q-1}^{\infty} \mathrm{E} |\widetilde{C}_{j}| \cdot n \mathrm{E} |Z_{n,1}^{>}|.\\[0.4em]
\end{equation*}
By Karamata's theorem and relation (\ref{eq:niz}), as $n \to \infty$,
$$ n \mathrm{E}|Z_{n, 1}^{>}| = \frac{\mathrm{E} ( |Z_{1}| 1 _{\{ |Z_{1}| > a_{n} \}})}{a_{n} \Pr (|Z_{1}| > a_{n})} \cdot n \Pr(|Z_{1}| > a_{n})  \to \frac{\alpha}{\alpha-1} < \infty,$$
and hence we see that there exists a positive constant $D_{2}$ such that
\begin{equation}\label{e:I2new}
  \limsup_{n \to \infty} I_{2} \leq D_{2} \sum_{j=q-1}^{\infty} \mathrm{E}|C_{j}|
\end{equation}
In the case $\alpha =1$ Markov's inequality implies
\begin{equation*}
 I_{2} \leq  \sum_{i=1}^{n} \Pr \bigg( \bigg| \sum_{j=q-1}^{\infty} \widetilde{C}_{j}Z_{n,i-j}^{>} \bigg| > \frac{\epsilon}{2} \bigg)
  \leq  \frac{2^{\delta}}{\epsilon^{\delta}}  \sum_{i=1}^{n} \mathrm{E} \bigg| \sum_{j=q-1}^{\infty} \widetilde{C}_{j}Z_{n,i-j}^{>} \bigg|^{\delta}
\end{equation*}
with $\delta$ as in relation (\ref{e:momcondr}). Since $\delta < 1$, an application of the triangle inequality $|\sum_{i=1}^{\infty}a_{i}|^{s} \leq \sum_{i=1}^{\infty}|a_{i}|^{s}$ with $s \in (0,1]$ yields
$$I_{2} \leq  \frac{2^{\delta}}{\epsilon^{\delta}} \sum_{i=1}^{n} \sum_{j=q-1}^{\infty} \mathrm{E} | \widetilde{C}_{j}Z_{n,i-j}^{>}|^{\delta} = \frac{2^{\delta}}{\epsilon^{\delta} } \sum_{j=q-1}^{\infty} \mathrm{E} |\widetilde{C}_{j}|^{\delta} \cdot n \mathrm{E}|Z_{n,1}^{>}|^{\delta}.$$
From this, since by Karamata's theorem
$$ \lim_{n \to \infty} n \mathrm{E}|Z_{n,1}^{>}|^{\delta} = \frac{n}{a_{n}^{\delta}} \mathrm{E} \Big( |Z_{1}|^{\delta} 1 _{\{ |Z_{1}|>a_{n} \}} \Big) = \frac{1}{1-\delta} < \infty,$$
and by a new application of the triangle inequality
\begin{eqnarray*}
 \sum_{j=q-1}^{\infty} \mathrm{E} |\widetilde{C}_{j}|^{\delta} & = & \mathrm{E} |\widetilde{C}_{q-1}|^{\delta} + \mathrm{E} |\widetilde{C}_{q}|^{\delta} + \sum_{j=q+1}^{\infty} \mathrm{E} |\widetilde{C}_{j}|^{\delta}\\[0.2em]
  & \leq & \mathrm{E} \bigg( 2 \sum_{j=q-1}^{\infty} |C_{j}| \bigg)^{\delta} + \mathrm{E} \bigg( 2 \sum_{j=q-1}^{\infty} |C_{j}| \bigg)^{\delta} + \sum_{j=q+1}^{\infty} \mathrm{E} |C_{j}|^{\delta}\\[0.2em]
  & \leq &  2^{\delta} \sum_{j=q-1}^{\infty} \mathrm{E} |C_{j}|^{\delta} + 2^{\delta} \sum_{j=q-1}^{\infty} \mathrm{E} |C_{j}|^{\delta} + \sum_{j=q+1}^{\infty} \mathrm{E} |C_{j}|^{\delta},
 \end{eqnarray*}
it follows that there exists a positive constant $D_{3}$ such that
\begin{equation*}\label{eq:I2new2}
 \limsup_{n \to \infty} I_{2} \leq D_{3} \sum_{j=q-1}^{\infty} \mathrm{E} |C_{j}|^{\delta}.
\end{equation*}
This together with (\ref{eq:I1I2}), (\ref{e:I1new2}) and (\ref{e:I2new}) yields
$$ \limsup_{n \to \infty}\Pr[d_{M_{1}}(M_{n}, M_{n,q})> \epsilon] \leq D_{1} \sum_{j=q-1}^{\infty} \mathrm{E} |C_{j}|+ (D_{2}+D_{3}) \sum_{j=q-1}^{\infty} \mathrm{E}|C_{j}|^{s},$$
 where
 $s=\delta$ for $\alpha=1$ and $s=1$ for $\alpha >1$.
 Now, conditions (\ref{e:momcondr}) and (\ref{eq:infmaTK3}) yield (\ref{e:Slutskyinf01}), and hence we again obtain $M_{n}(\,\cdot\,) \dto M(\,\cdot\,)$ in $(D_{\uparrow}[0,1], d_{M_{1}})$. This concludes the proof.
\end{proof}

\begin{rem}
Since $(C^{(1)}, C^{(2)})$ is independent of $(W^{(1)}, W^{(2)})$, the limiting process $M$, defined in (\ref{e:limprocess}) by
$$  M(t) = C^{(1)}W^{(1)}(t) \vee C^{(2)}W^{(2)}(t) =  \bigvee_{t_{i} \leq t}|j_{i}| ( C^{(1)}1_{\{ j_{i} >0 \}} + C^{(2)} 1_{\{ j_{i}<0 \}}), \qquad t \in [0,1],$$
where $\sum_{i}\delta_{(t_{i},j_{i})}$ is a Poisson process with intensity measure $\emph{Leb} \times \mu$, with $\mu$ as in (\ref{eq:mu}),
 conditionally on $(C^{(1)}, C^{(2)})=(a,b)$, is an extremal process with exponent measure $(pa^{\alpha} + rb^{\alpha}) \alpha x^{-\alpha-1} \rmd x$ for $x>0$ and non-negative real numbers $a$ and $b$. Indeed, for $x>0$ we have
\begin{eqnarray}\label{e:limMconditioning}
 \nonumber \Pr[ M(t) \leq x\,|\,(C^{(1)}, C^{(2)})=(a,b) ] &=& \Pr \Big( \bigvee_{t_{i} \leq t}|j_{i}| ( a 1_{\{ j_{i} >0 \}} + b 1_{\{ j_{i}<0 \}}) \leq x \Big)\\[0.4em]
 & = & \Pr \Big( \bigvee_{t_{i} \leq t}j_{i}S_{i} \leq x \Big)
\end{eqnarray}
  with $S_{i} = a1_{\{ j_{i} >0 \}} - b 1_{\{ j_{i}<0 \}}$.
 Propositions 3.7 and 3.8 in Resnick~\cite{Re87} yield that $\sum_{i} \delta_{(t_{i}, j_{i}, S_{i})}$ is a Poisson process with intensity measure $\emph{Leb} \times \widetilde{\mu}$, where
  $\widetilde{\mu}(dx, dy)= \mu(dx) K(x, dy)$ and
  $ K(x, dy) = \Pr \big( a1_{\{ x>0 \}} - b 1_{\{ x<0 \}} \in\,dy \big),$
 and that $\sum_{i}\delta_{(t_{i}, j_{i}S_{i})}$ is a Poisson process with intensity measure $\emph{Leb} \times \widehat{\mu}$, where
  $$ \widehat{\mu}(x, \infty) = \widetilde{\mu}(\{(y,z) : yz > x \}) = \int\!\!\!\int_{yz>x} \mu(dy) K(y, dz), \qquad x >0.$$
  From this we conclude that the process $ \bigvee_{t_{i} \leq \cdot}j_{i}S_{i}$
is an extremal process with exponent measure $\widehat{\mu}$ (see Resnick~\cite{Re87}, Section 4.3; and Resnick~\cite{Resnick07}, p.~161). Standard computations give
  \begin{equation*}
  \widehat{\mu}(x, \infty) =   \int_{0}^{\infty} \int_{x/y}^{\infty} K(y,dz)\mu(dy) + \int_{-\infty}^{0} \int_{-\infty}^{x/y} K(y,dz)\mu(dy)  = px^{-\alpha}a^{\alpha} + r x^{-\alpha} b^{\alpha}.
  \end{equation*}
Hence
$$ \widehat{\mu}(dx) = (pa^{\alpha} + r b^{\alpha}) \alpha x^{-\alpha-1}1_{(0,\infty)}(x)\,dx,$$
and we conclude from (\ref{e:limMconditioning}) that the limiting process $M$, conditionally on $(C^{(1)}, C^{(2)})=(a,b)$, is an extremal process with exponent measure $\widehat{\mu}$.

Note that, conditionally on $\{C_{j}=c_{j} \ \textrm{for all} \ j \geq 0\}$, where $(c_{j})_{j}$ is a sequence of real numbers, the process $X_{i}$ in (\ref{e:MArandom}) is a linear process with deterministic coefficients $(c_{j})$.
Therefore,
Proposition 4.28 in Resnick~\cite{Re87} yields that
the limit of $M_{n}$ in $D_{\uparrow}[0,1]$ with the $M_{1}$ topology is a process which is, conditionally on $\{C_{j}=c_{j} \ \textrm{for all} \ j \geq 0\}$, an extremal process with exponent measure $(pc_{+}^{\alpha} + rc_{-}^{\alpha}) \alpha x^{-\alpha-1} \rmd x$ for $x>0$, provided $c_{+}p + c_{-}r>0$, where $c_{+}= \max\{c_{j} \vee 0 : j \geq 0\}$ and $c_{-}= \max\{ - c_{j} \vee 0 : j \geq 0\}$. This corresponds to the above considerations about the structure of the limiting process for $a=c_{+}$ and $b=c_{-}$.
\end{rem}

\begin{rem}
 If the sequence $(C_{j})$ is deterministic, condition (\ref{e:mod1}) can be dropped since it is implied by (\ref{e:momcondr}). Note that condition (\ref{e:momcondr}) implies $|C_{j}|^{\delta} < 1$ for large $j$, and since $|C_{j}|^{\delta x}$ is decreasing in $x$, it follows that for large $j$
$$ |C_{j}|^{\gamma} = (|C_{j}|^{\delta})^{\gamma/\delta} \leq |C_{j}|^{\delta}.$$
This suffices to conclude that (\ref{e:mod1}) holds.
 In general this does not hold when the coefficients are random (see Krizmani\'{c}~\cite{Kr19}, p.~739).
\end{rem}

\section*{Acknowledgements}
 This work has been supported in part by University of Rijeka research grants uniri-prirod-18-9 and uniri-pr-prirod-19-16 and by Croatian Science Foundation under the project IP-2019-04-1239.


\end{document}